\documentclass[11pt]{amsart}
\usepackage{ragged2e}
\usepackage{graphicx,subfig}
\usepackage{amsmath}
\usepackage{amsthm}
\usepackage{amsfonts}
\usepackage{epsfig}
\usepackage{indentfirst}
\usepackage{psfrag,epsf}
\usepackage{amssymb, textcomp}
\usepackage[all]{xy}
\usepackage[hypertex]{hyperref}

\usepackage{multicol}

\newtheorem{thm}{Theorem}[section]

\newtheorem{lemma}[thm]{Lemma}
\newtheorem{defn}[thm]{Definition}
\newtheorem{prop}[thm]{Proposition}
\newtheorem{cor}[thm]{Corollary}
\newtheorem{conje}[thm]{Conjecture}

\newtheorem{rem}[thm]{Remark}

\numberwithin{equation}{section}

\newtheorem{que}[thm]{Question}
\newtheorem{prob}[thm]{Problem}

\def\x{\mathrm{x} }
\def\y{\mathrm{y}}

\def\Z{\mathbb{Z}}
\def\R{\mathbb{R}}

\def\M{\mathrm{M}}

\def\W{\mathrm{W}}

\def\Barint_#1{\mathchoice
         {\mathop{\vrule width 6pt height 3 pt depth -2.5pt
                 \kern -8pt \intop}\nolimits_{#1}}%
         {\mathop{\vrule width 5pt height 3 pt depth -2.6pt
                 \kern -6pt \intop}\nolimits_{#1}}%
         {\mathop{\vrule width 5pt height 3 pt depth -2.6pt
                 \kern -6pt \intop}\nolimits_{#1}}%
         {\mathop{\vrule width 5pt height 3 pt depth -2.6pt
                 \kern -6pt \intop}\nolimits_{#1}}}

\renewcommand{\(}{\left(}
\renewcommand{\)}{\right)}

\newcommand{\diam}{diam}
\newcommand{\dist}{dist}

\begin{document}


\title[A Characterization of bi-Lipschitz embeddable metric spaces 
]{A Characterization of bi-Lipschitz embeddable metric spaces in terms of local bi-Lipschitz embeddability}

\author[J. Seo]{Jeehyeon Seo}
\thanks{This work was supported by US National Science Foundation Grant  DMS-0901620.}

\address
{Department of Mathematics, University of Illinois at Urbana Champaign\\
1409 West Green Street, Urbana, IL, 61801}
\email {seo6@illinois.edu}

\begin{abstract}
We characterize uniformly perfect, complete, doubling metric spaces which embed bi-Lipschitzly into Euclidean space. Our result applies in particular to spaces of Grushin type equipped with Carnot-Carath\'{e}odory distance. Hence we obtain the first example of a sub-Riemannian manifold admitting  such a bi-Lipschitz embedding. Our techniques involve a passage from local to global information, building on work of Christ and McShane. A new feature of our proof is the verification of the co-Lipschitz condition. This verification splits into a large scale case and a local case. These cases are distinguished  by a relative distance map which is associated to a Whitney-type decomposition of an open subset $\Omega$ of the space. We prove that if the Whitney cubes embed uniformly bi-Lipschitzly into a fixed Euclidean space, and if the complement of $\Omega$ also embeds, then so does the full space.
\end{abstract}


\maketitle


\section{Introduction}
A map between two metric spaces is bi-Lipschitz if distances in the image  and  source should not exceed distances in the source and image respectively by more than a fixed, universal multiplicative constant. More precisely, a map $f$ between metric spaces $(X,\,d_X)$ and $(Y,\, d_Y)$ is called bi-Lipschitz  if there exists  an  $L \geq1$ such that
\begin{equation}
\label{BL}
\frac{1}{L}\, d_X(x,\, y)\leq d_Y(f(x),\, f(y)) \leq L\,d_X(x,\,y)
\end{equation}
for all $x,\, y \in X$.

Bi-Lipschitz maps play a role in computer science as well as in many branches of mathematics. Solving the Sparsest cut problem approximately is  important  in the theory of approximation algorithms. The best known algorithm for this question is related to the Goemans-Linial conjecture \cite{Goemans, Linial2}. Recently, Cheeger and Kleiner \cite{Jeff} together with Lee and Naor \cite{Naor} gave an counterexample to the Goemans-Linial conjecture. They showed that the Heisenberg group admits a metric which is of negative type, yet  does not admit a bi-Lipschitz embedding into $L^1$.

Bi-Lipschitz maps are related to problems of differentiability by Rademacher's theorem. Lipschitz maps form the right substitute for smooth maps in the theory of analysis on metric spaces. We would like to know for which metric spaces the resulting analysis is genuinely new and for which ones the analysis can be seen as just classical analysis on a suitable subset of a Banach space. This leads to the question to characterize metric spaces that embed bi-Lipschitzly into classical Banach spaces. However, the characterization of metric spaces which are bi-Lipschitz equivalent to $\R^n$  or even of metric spaces which are bi-Lipschitzly embeddable into $\R^n$ remain  difficult open problems in Geometric Analysis.

 We are interested in the question of which metric spaces embed bi-Lipschitzly into Euclidean space. We state some progress on this problem. Assouad gave a partial answer: every snowflaked version of a doubling metric space embeds bi-Lipschitzly into some Euclidean space \cite{Assouad}. Even though  the theorem of Assouad completely answers the question which metric spaces are quasisymmetrically embeddable into Euclidean space, this result does not guarantee bi-Lipschitz  embeddability of the original metric space.  In particular, the Heisenberg group, which is a doubling metric space,  admits no  bi-Lipschitz embedding into Euclidean space. 
 Luosto \cite{Luosto} together with  Luukkainen and Movahedi-Lankarani \cite{Luukkainen} gave a precise relationship between Assouad dimension and dimension of receiving Euclidean space  for ultra metric spaces: an ultrametric space is bi-Lipschitzly embeddable into $\mathbb R^n$ if and only if its Assouad dimension is less than $n$.

Semmes \cite{Semmes2} showed that $\mathbb R^n$ equipped with  any metric $\delta_{\omega}$ deformed by $A_1$-weight $\omega$ admits a bi-Lipschitz embedding into some $\mathbb R^N$. However, $(\R^n, \delta_{\omega})$ may be not bi-Lipschitzly equivalent to $\R^n$. Bishop \cite{Bishop2} constructed a Sierpinski carpet $E\subset \R^2$ and an $A_1$-weight $\omega$ which blows up on $E$. In this construction, he showed that $w$ is not comparable to the Jacobian of any quasiconformal mapping. 

In this paper, we will characterize uniformly perfect complete metric spaces which admit a bi-Lipschitz embedding in terms of uniform local bi-Lipschitz embeddability. Indeed, uniform perfectness and existence of a doubling measure yield existence of a  Whitney-type decomposition. Furthermore, uniform local bi-Lipschitz embeddability of Christ cubes associated with such a decomposition implies global bi-Lipschitz embeddability.
\begin{thm}
\label{main1}
A uniformly perfect complete metric space $(X, d)$ admits a bi-Lipschitz embedding into some Euclidean space if and only if the following conditions hold:
\begin{enumerate}
\item it supports a doubling  measure $\mu$,
\item there exists a closed subset $Y$ of $X$ which admits a bi-Lipschitz embedding into some $\R^{M_1}$, 
\item $\Omega = X\setminus Y$ admits uniformly Christ-local bi-Lipschitz embeddings into some $\R^{M_2}$. 
\end{enumerate}
The bi-Lipschitz constant and dimension of receiving Euclidean space depend on the data of the metric space $X$, the doubling constant of $\mu$,  $M_1$, $ M_2$ and the bi-Lipschitz constants in conditions $(2)$ and $(3)$.
\end{thm}

 We now discuss applications of Theorem \ref{main1} to the bi-Lipschitz embedding question for sub-Riemannian manifolds. For more information of Carnot-Carath\'{e}odory geometry, see \cite{Gromov}. Pansu \cite{Pansu} showed that a version of Rademacher's  differentiation theorem holds for Lipschitz maps on Carnot groups: every Lipschitz map between Carnot groups is almost everywhere differentiable in some sense and its differential is a Lie group homomorphism.  Semmes  observed  that Pansu's result implies that nonabelian Carnot groups admit no bi-Lipschitz embedding into Euclidean space (Theorem 7.1 in \cite{Semmes1}). Cheeger proved a remarkable extension of  Rademacher's theorem for doubling $p$-Poincar\'e spaces  and gave a corresponding  nonembedding theorem (see Section 10 and Theorem 14.3 in \cite{Cheeger1}). 

By Cheeger's theorem, we can deduce nonembeddability of certain regular sub-Riemannian manifolds. However, his result does not apply to singular sub-Riemannian manifolds. This paper is motivated by the question whether or not the Grushin plane embeds bi-Lipschitzly into Euclidean space. While the Grushin plane is one of the simplest singular sub-Riemannian manifold, the previous known nonembedding theorems do not apply.  In  contrast,  as an application of Theorem \ref{main1}  we will prove bi-Lipschitz embeddability of the Grushin plane.  This is  the first example of a sub-Riemannian manifold that embeds bi-Lipschitzly into Euclidean space. 
\begin{defn}
The Grushin plane $\mathbb G$ is $\R^2$ with horizontal distribution spanned by 
\begin{equation*}
X_1=\dfrac{\partial}{\partial x} \;\;\; \text{and} \;\;\; X_2=x\,\dfrac{\partial}{\partial y}.
\end{equation*}
\end{defn}
\begin{thm}
The  Grushin plane equipped with Carnot-Carath\'{e}odory distance admits a bi-Lipschitz embedding into some Euclidean space.
\end{thm}
 
The structure of this paper follows. In the second section, we shall see Assouad's embedding theorem and Lipschitz extension theorem.  We will review Michael Christ's construction of a system of dyadic cubes \cite{Michael} in doubling metric spaces. We will next construct a Whitney-type decomposition  which we call a Christ-Whitney decomposition (Lemma \ref{W-type}) for a uniformly perfect space supporting a doubling measure. We will also introduce some definitions and lemmas which set the stage for  Theorem \ref{main1}.

In the following section,  we shall characterize bi-Lipschitz embeddable metric spaces by proving Theorem \ref{main1}.  To this end, we first apply McShane's  extension theorem to extend a Lipschitz map on $Y$ to $X$. We introduce the Whitney distance map $d_\mathrm{W}$ (Definition \ref{W-distance}). It is the key tool for construction of a co-Lipschitz map. We break the Christ-Whitney decomposition into two parts using the Whitney distance map. After some basic preliminaries, we will construct a W-local co-Lipschitz and  W-large scale co-Lipschitz map on these parts (Lemma \ref{W-lem1} and  Lemma \ref{W-lem2}).

In Section 3, we discuss applications of Theorem \ref{main1} to the bi-Lipschitz embedding question for sub-Riemannian manifolds.  We prove bi-Lipschitz embeddability of the Grushin plane into some Euclidean space.

\section*{Acknowledgments}
I am grateful to my thesis adviser, Jeremy T. Tyson for suggesting the problem and for many stimulating conversations.  I am also greatly indebted to John Mackay.  The distinction between the W-large and  the W-local co-Lipschitz conditions was suggested by Prof. Mackay when I took reading courses in Fall 2009 and in  Spring 2010. I also want to thank Jang-Mei  Wu, John P. D'Angelo, and  Sergiy Merenkov for their careful reading and comments.

\section{Preliminaries}
\subsection{Notation and Terminology}
For  a metric space $X= (X, \,d)$, we write $\diam(A)$ (or $\diam_d(A)$ in case we need to mention the metric) for the diameter of a set $A \subset X$, and $\dist(A, B)$ for the distance between nonempty sets $A,\; B \subset X$. We abbreviate $\dist(A, \,x) =\dist(A, \,\{x\})$ for a set $A\subset X$ and $x\in X$. We write $d_E$ for the Euclidean metric. As customary, we let $C,\,c,\cdots$ denote finite positive constants. These constants may depend on auxiliary data $a,\, b$, etc ; we indicate this by writing $C(a,\,b)$ or $c(a,\,b)$. We also write $a\lesssim b$ if there is a constant $C$ such that $a\leq C\, b$.

We recall that the map $f:X \rightarrow Y$ is a bi-Lipschitz embedding if \eqref{BL} holds. We do not assume that $f$ is onto.  We say an invertible map $f: X\rightarrow Y$ between metric spaces is co-Lipschitz if $f^{-1}$ is Lipschitz. We call  any constant  $L$ satisfying \eqref{BL} a bi-Lipschitz constant for $f$.
\begin{defn}
A metric space $(X, \,d)$ is uniformly perfect if there exists a constant $A>0$ such that for each $x \in X$ and $0<r < \diam X$ there is a point $y \in X$ which satisfies $A^{-1}r\leq d(x, y) \leq r$. We say that $(X,\,d)$ is $A$-uniformly perfect.
\end{defn}
Uniform perfectness implies nonexistence of separating annuli of large modulus and nonexistence of isolated points.  Every connected metric space is uniformly perfect. In an $A$-uniformly  perfect space, $\overline{B}(x,\, r) \setminus B(x,\, A^{-1}r)$ is nonempty for all $x\in X$ and $0<r<\diam X$ and so $ A^{-1} r\leq \diam B(x,\, r) \leq 2r$.

The doubling condition provides a kind of boundedness of the geometry of the space. 
\begin{defn} A Borel measure $\mu$  in a metric space is called doubling if balls have finite and positive measure for any nonempty ball and there is a constant $D\geq 1$ such that 
\begin{equation}
\label{doubling}
\mu(B(x,2r)) \leq D\mu(B(x,r))
\end{equation}
for all $x \in X$ and $r>0$. We call $D$ a doubling constant.
\end{defn}
\begin{defn}
A metric space is called doubling if there is a constant $C$ so that every set of diameter $d$  in the space can be covered by at most $C$ sets of diameter at most $d/2$.
\end{defn}

\subsection{Basic Theorems}
In this section, we recall Assouad's embedding theorem and McShane's Lipschitz extension theorem. 
 \begin{thm}[Assouad \cite{Assouad}]
\label{Assouad}
Each snowflaked version of a doubling metric space admits a  bi-Lipschitz embedding into some Euclidean space. If $ 0<\epsilon<1$, then $(\mathbb R, \,{d_{E}}^{\epsilon})$ embeds bi-Lipschitzly into $\mathbb R^k$, where $k$ is the smallest integer which is greater than $\frac{1}{\epsilon}$.
\end{thm}
The identity snowflaking $(X, \,d) \rightarrow (X,\, d^{\epsilon})$ is $t^{\epsilon}$-quasisymmetric and hence, each metric space is quasisymmetrically embedded in Euclidean space if and only if it is doubling. However, Assouad's theorem does not answer whether or not the original metric space embeds bi-Lipschitzly. For example,  whereas the snowflaking of the Heisenberg group endowed with Carnot-Carath\'{e}odory distance, $(\mathbb H, \,{d_{cc}}^{\epsilon})$, 
admits a bi-Lipschitz embedding into some Euclidean space, the Heisenberg group is not bi-Lipschitzly embeddable into any Euclidean space. Such nonembeddability is a consequence of Pansu's Rademacher-type theorem \cite{Pansu} as observed by Semmes \cite{Semmes1}. It also follows from Cheeger's nonembedding theorem \cite{Cheeger1}. 

To prove Theorem \ref{Assouad}, Assouad builds a multiscale family of maps on scale $2^{-j}$ for each $j \in \Z$ and glues these maps together into an embedding using $2^{-j}$-nets and a coloring map. A similar idea will appear in the proof of Theorem \ref{mainlem}. In fact, we shall consider a Whitney-type decomposition instead of nets and use a coloring map to increase dimension of receiving Euclidean space.

With some restrictions on $X$ and $Y$, and  for $A \subset X$, 
every Lipschitz function $f:A \rightarrow Y$ can be extended to a Lipschitz function $F: X \rightarrow Y$.  We recall McShane's Lipschitz extension theorem. Since McShane's Lipschitz extension map has no restriction on the source space, it is useful for our purpose. For further information, see \cite{Juha1} \cite{Juha2}.
\begin{thm}[McShane]
\label{McShane}
Let $X$ be an arbitrary metric space.
If $A\subset X$ and $f :A \rightarrow \mathbb R$ is $L$-Lipschitz, then there exists an $L$-Lipschitz function $F : X \rightarrow \mathbb R$ which extends $f$.
i.e. $F|_{A}= f$.
\end{thm}
\begin{cor}[McShane]
\label{McShane2}
Let $f:A \rightarrow \mathbb R^M$ where $A\subset X$, be an $L$-Lipschitz function. Then, there exists an $\sqrt{M}L$-Lipschitz function $F:X \rightarrow \mathbb R^M$ such that $F|_{A}=f$.
\end{cor}
\subsection{Christ-Whitney Decomposition}
As Euclidean space has a system of dyadic cubes, every doubling metric measure space also has a system of sets akin to classical dyadic cubes. The following Proposition \ref{mchrist} may be transparent if we think of $Q_{\alpha}^{k}$ as being essentially a cube of diameter roughly $\delta^{k}$ with center $z_{\alpha}^k$.  When $Q_{\beta}^{k+1} \subset Q_{\alpha}^{k}$, we say that $Q_{\beta}^{k+1}$ is a child of $Q_{\alpha}^{k}$ and $Q_{\alpha}^{k}$ is a parent of $Q_{\beta}^{k+1}$.
\begin{prop}[Christ \cite{Michael}]
\label{mchrist}
Let $(X, d, \mu)$ be a doubling metric measure space.
Then, there exists a collection of open subsets $\{Q_{\alpha}^k \subset X \;| \; k \in \Z \;,\; \alpha\in I_k \} $ where $I_k$ is some index set depending on $k$,  and constants $\delta\in (0,\,1) \;, a_0 \in (0,\,1)$, $\eta > 0$ and $C_1,\; c <\infty$ such that
\begin{enumerate}
\item $\mu(X\setminus \cup_{\alpha\in I_k }Q_{\alpha}^{k} )=0 ,\;\; \text{for all}\;\; k \in \Z$.
\item For any $\alpha$, $\beta$, $k$, and $l$ with $l \geq k$,  either $ Q_{\beta}^{l} \subset Q_{\alpha}^{k}$ or $Q_{\beta}^{l} \bigcap Q_{\alpha}^{k}=\emptyset$.
\item Each $Q_{\alpha}^{k}$ has exactly one parent and at least one child for all $k\in \Z$.
\item For each  $\(\alpha, k\)$, there exists $ z_{\alpha}^k \in X $such that 
$ B(z_{\alpha}^k,\, a_0\delta^k) \subset Q_{\alpha}^k \subset B(z_{\alpha}^k, \,C_1\delta^k)$.
 \end{enumerate}

\end{prop}
We now introduce a Whitney-type decomposition on an open subset of a uniformly perfect  metric space supporting a doubling measure. As open subset of Euclidean space has a Whitney decomposition from a system of dyadic  cubes, we have a Whitney-type decomposition from a system of Christ cubes. We call it a Christ-Whitney decomposition. This decomposition has a comparability condition (see $(4)$ Lemma \ref{W-type}) in addition to all conditions of a Whitney decomposition. This comparability condition together with doubling condition will play an important role in the proof of Lemma \ref{number}, which yields existence of a coloring map in Lemma \ref{color}.
\begin{lemma}
\label{W-type}
Suppose that $(X, d, \mu)$ is a $A$-uniformly perfect metric space supporting a doubling metric measure,  $Y$ is a closed subset of $X$, and  $\Omega = X\setminus Y$. Then $\Omega$ has a Christ-Whitney decomposition $\M_{\Omega}$  satisfying the following properties:
\begin{enumerate}
\item $\mu(\Omega \setminus \cup_{Q\in \M_{\Omega}}Q)=0$.
\item  $\diam(Q)\leq \dist(Q,\, Y) \leq \dfrac{4C_1\,A}{\delta}\,\diam(Q)$.
\item  $Q\cap Q^{'}= \emptyset$.
\item For any $Q \in \M_{\Omega}$, there exists $x \in \Omega$ such that
\begin{equation*}
 (*)\; B(x, \,a_0\delta^k) \subset Q \subset B(x, \, C_1\delta^k)
 \end{equation*} 
 for some $k $.
\end{enumerate}
The constants $\delta$, $a_0$ and $C_1$ are deduced from Proposition \ref{mchrist}. 
\end{lemma}
\begin{rem}
We say that $Q$ is $(C_1, \,a_0)$-quasiball if  $(*)$ holds for some $x$ and $\delta$.
From now on, we will call a ball $B(x, \,C_1\delta^k)$ containing $Q$ a $C_1$-quasiball of $Q$ and denote it by $\widetilde{B_Q}$. We observe that $\diam(\widetilde{B_Q})$ is comparable to $\delta^k$ by uniform perfectness of $X$.
\end{rem}
\begin{proof}
Since $\Omega =X\setminus Y$ is a doubling metric measure space, we have a family of subsets 
\begin{equation*}
\{Q_{\alpha}^k \subset \Omega \;| \; k \in \Z \;,\; \alpha\in I_k \}
\end{equation*}
for fixed constants  $\delta$ and  $C_1$ so that $\mu(\Omega \setminus  \cup_{\alpha \in I_k}Q_{\alpha}^{k} )=0$ from Proposition \ref{mchrist}.
We now consider layers, defined by $\Omega_k = \{x \;|\; c'\delta ^k < \dist(x, \,Y) \leq c' \delta ^{k-1}\}$, where $c'$ is a positive constant we shall fix momentarily.
Obviously, $\Omega = \cup_{k=-\infty}^{\infty} \Omega_k $.

We now make an initial choice of $Q$'s, and denote the resulting collection by $\M_0$. Our choice is made as follows. We consider $Q$'s chosen from $ \mathcal {A}^k=\{Q_{\alpha}^k \;| \;{\alpha} \in I_k \} $ for each $k\in \Z$, (each such $Q$ is of  size approximately $\delta^k$), and include a $Q$ in $\M_0$ if it intersects $\Omega_k$. In other words,
\begin{equation*}
\M_0 = \cup_{k}\{Q \in \mathcal{A}^k \;|\; Q \cap \Omega_{k} \neq \emptyset \}.
\end{equation*}

We then have $\mu(\Omega \setminus \cup_{Q\in \M_0} Q)= 0$.
For an appropriate choice of $c'$,
\begin{equation}
\label{whty1}
\diam(Q)\leq \dist(Q, \,Y) \leq \dfrac{4C_1\,A}{\delta}\,\diam(Q).
\end{equation}
Let us prove \eqref{whty1} first. Suppose $Q \in \mathcal {A}^k$, then $\dfrac{1}{A}\delta ^k \leq \diam(Q)\leq 2C_1\delta ^k$ because of uniform perfectness. Since $Q \in \M_0$, there exists  $x \in Q\cap \Omega_k$. Thus, $\dist(Q, Y) \leq \dist(x, Y)\leq c' \delta ^{k-1} \leq \dfrac{c'A}{\delta}\, \diam Q \leq \dfrac{4C_1\,A}{\delta}\,\diam(Q)$ and  $\dist(Q, \,Y) \geq \dist(x,\, Y) -\diam(Q) \geq c'\delta^k-2C_1\,\delta^k =2C_1\,\delta^k \geq \diam(Q)$. If we choose $c'=4C_1$, we get the equation \eqref{whty1}.

Notice that the collection $\M_0$ has all required properties, except that $Q$'s in it are  not necessarily disjoint. To finish the proof of the lemma we need to refine our choice leading to $\M_0$, eliminating $Q$'s which were really unnecessary. We require the following observation. Suppose $Q \in \mathcal {A}^k$ and $Q' \in \mathcal{A}^{k'}$. If $Q$ and $Q'$ are not disjoint, then one of two must be contained in the other. Start now with any $Q \in  \M_0$, and consider the unique maximal parent in $\M_0$ which contains it. We let $\M_{\Omega}$ denote the collection of maximal $Q$'s in $\M_0$.  The last property comes straightforward from Proposition \ref{mchrist} and Lemma \ref{W-type} is therefore proved.
\end{proof} 

We now define new concepts $Q^*$ and $Q^{**}$ corresponding to a Whitney cube $Q$ and a dilated Whitney cube $\lambda Q$  respectively in the classical Whitney decomposition. 
  \begin{defn}
 \label{Q^*}
 For any fixed $Q \in \M_{\Omega}$, we denote  by $Q^*$ the collection of all  $R \in\M_{\Omega}$ whose distance from $Q$ does not exceed minimum diameters of $R$ and $Q$ by a fixed constant $\epsilon$.  We denote by $Q^{**}$ the collection of all $S\in \M_{\Omega}$ whose distance from some $R \in Q^*$ does not exceed minimum diameters of $R$ and $S$ by a fixed constant $\epsilon$. Here $\epsilon$ is a fixed number such that $0<\epsilon<1$. In other words,
 \begin{enumerate}
 \item $Q^* =\cup \{R \in \M_{\Omega} \;|\; \dist(Q,\, R) <\epsilon \min\{\diam(Q), \, \diam(R)\}\;\}$.
 \item $Q^{**}=\cup \{S \in \M_{\Omega} \;|\;\dist(S,\, R) <\epsilon \min\{\diam(S), \, \diam(R)\}\;\text{for some} \; R \in Q^* \;\} $.
  \end{enumerate}
 \end{defn}
 \begin{rem}
  $Q^*$ could contain no other Christ-cubes except $Q$. Throughout this paper, we can choose any $\epsilon$. However, in practice, we will restrict $\epsilon$ to a universal fixed number  in $(0,\,1)$ since we will consider condition of  uniformly Christ-local bi-Lipschitz embeddings (Definition \ref{measurelocal}).
 \end{rem}
 \begin{rem}
 Figure \ref{fig1}, Figure \ref{fig2} and Figure \ref{fig3} illustrate an idea how our construction goes. Of course, actual shapes will depend on a metric space. 
  \end{rem}
 
 We next see some propositions related to $Q^*$ and $Q^{**}$.
 \begin{prop}
 \label{comparability}
 For any fixed $Q \in \M_{\Omega}$, suppose $R \in Q^*$. Then,
\begin{equation*}
 [\dfrac {  4C_1\,A}{\delta}+1+\epsilon]^{-1} \,\diam(R) \leq \diam(Q)   \leq [\dfrac {  4C_1\,A}{\delta}+1+\epsilon] \,\diam(R)
 \end{equation*}
 \end{prop}

  \begin{figure}[h]
 \centering\subfloat[Christ-Whitney decomposition $\M_{\Omega}$]
  {\includegraphics[width = 6cm]{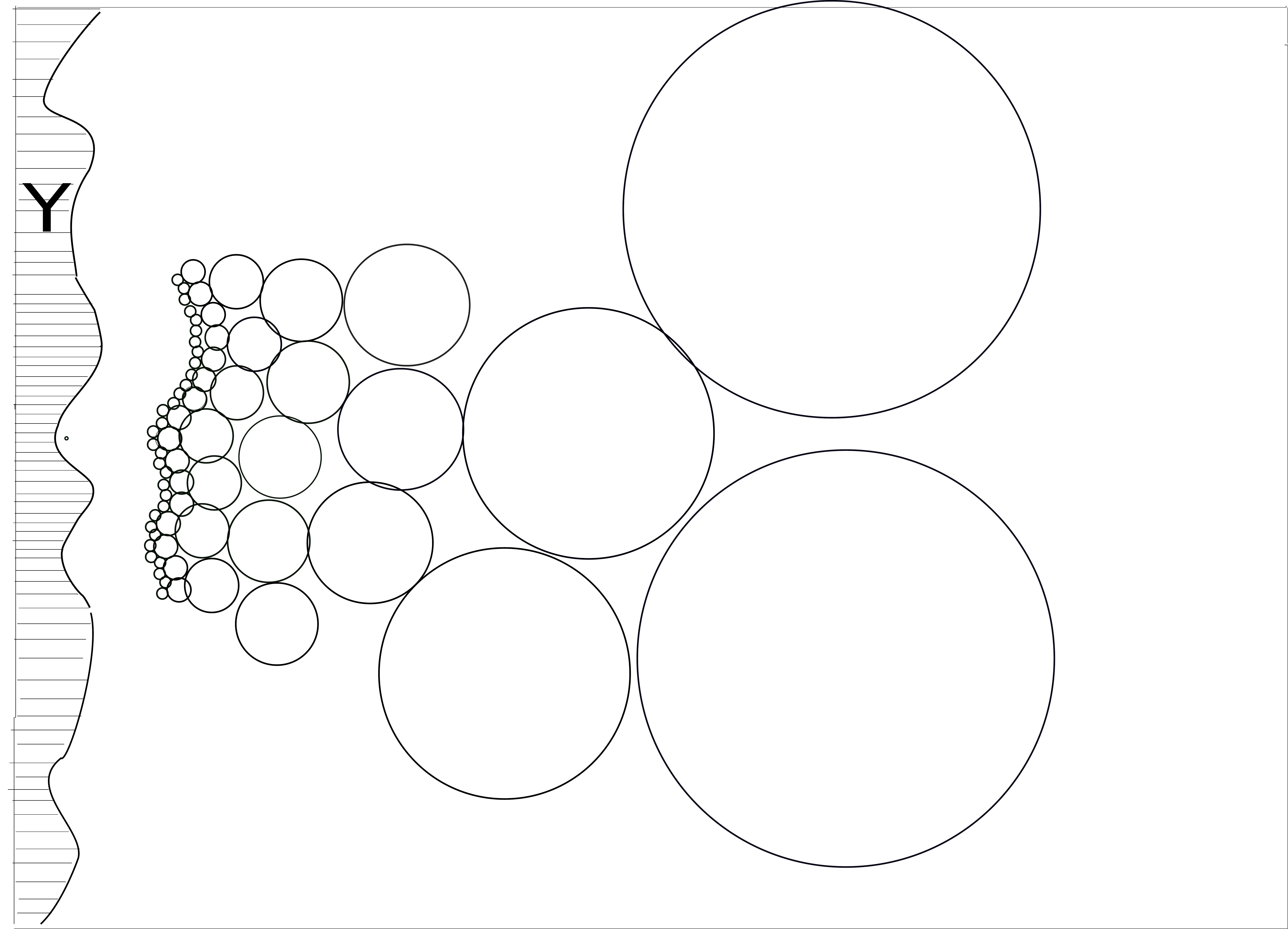}}\hspace*{3cm}
 \subfloat[Definition of $Q^*$ and $Q^{**}$]
  {\includegraphics[width = 6cm]{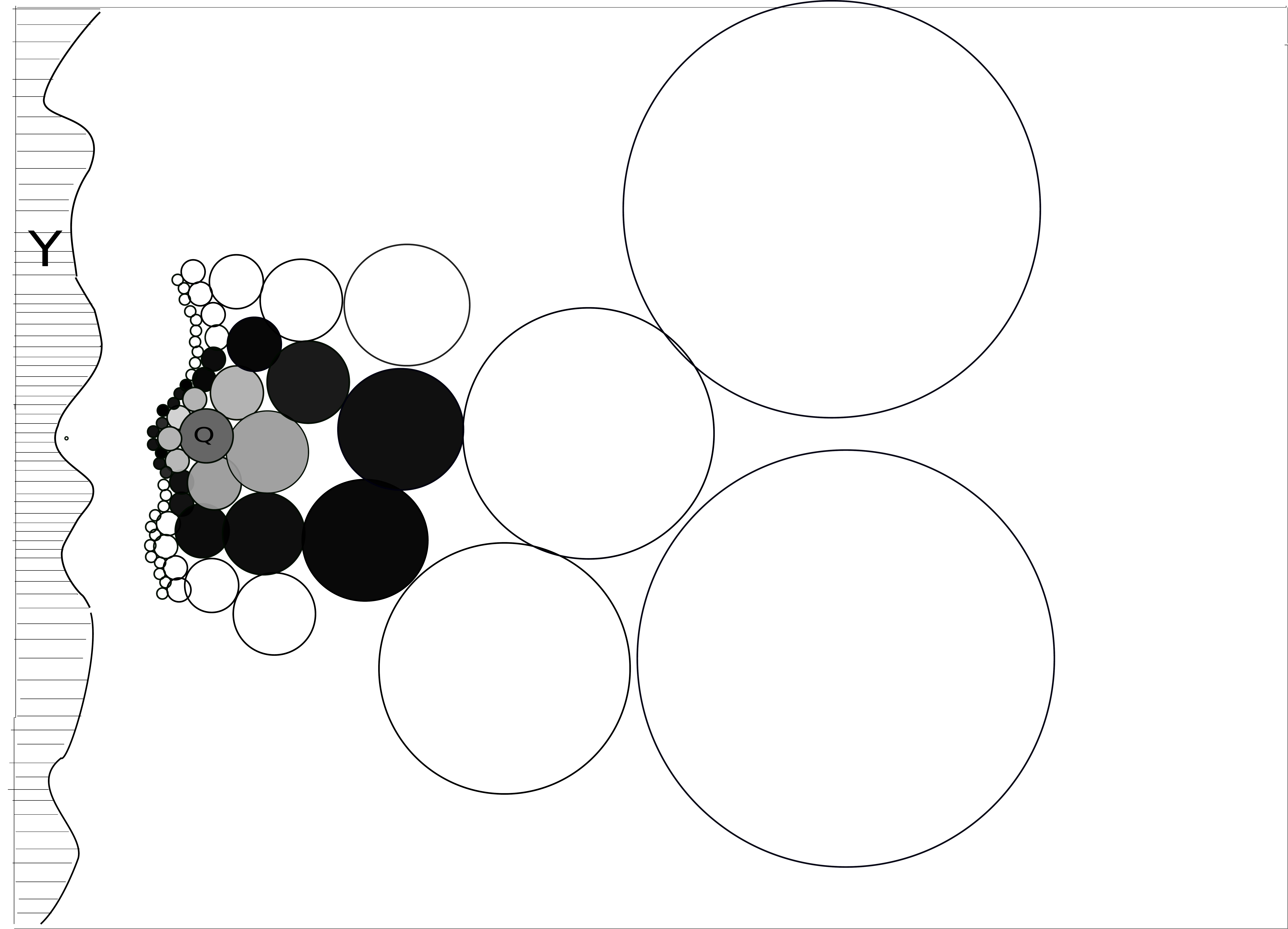}}
 \caption{The gray balls are elements of $Q^*$ and  gray and black balls are elements of $Q^{**}$}
 \label{fig1}
\end{figure}

 \begin{proof}
 We suppose that $\diam(R) \geq \diam(Q)$. Then, we arrive at 
\begin{align*}
\diam(R) &\leq \dist(R,\, Y)\\
              &\leq \diam(Q)+\dist(Q,\, Y) +\dist(R,\, Q) \\
               &\leq [\dfrac{4C_1\,A}{\delta}+1+\epsilon]\,\diam(Q)
\end{align*}
 and the symmetrical implication proves the proposition. 
 \end{proof}
 
 \begin{prop}
  \label{touch}
  Let $(X,\,d)$ be a uniformly perfect metric space supporting a doubling measure $\mu$.
   and let $\M_{\Omega}$ be a Christ-Whitney decomposition as in Lemma \ref{W-type}.
  \begin{enumerate}
 \item Suppose $Q \in \M_{\Omega}$. Then there are at most $N$ Christ cubes in $\M_{\Omega}$ in $Q^{**}$. 
  \item Any point in  $\M_{\Omega}$ is contained in at most $N$ of $Q^{**}$.
 \end{enumerate}
 The number $N$ is independent of $Q$. It depends on the doubling constant of $\mu$, $\epsilon$ and the data of $X$. 
  \end{prop}

 \begin{proof}
For any $R \in Q^{**}$,  we have comparability between $\diam(Q)$ and  $\diam(R)$ from Proposition \ref{comparability}. Therefore, $\diam(Q^{**})$ is comparable to $\diam(Q)$. Doubling condition yields that there are at most a finite number of such $R$'s and hence there are at most $N(\mu, \,C_1,\,A,\, \delta,\,\epsilon)$ Christ cubes in $Q^{**}$.

Let $p$ be a point in $\M_{\Omega}$ and write $p \in R$. We now observe that  for any $Q \in R^{**}$, we have $R\in  Q^{**}$. We have $p \in Q^{**}$ for all $Q  \in R^{**}$ and hence $p$ is contained in at most $N$ sets of type $Q^{**}$ by Proposition \ref{touch} $(1)$.
 \end{proof}

We now build a family of Lipschitz cutoff functions.  We will use these functions to construct a W-local co-Lipschitz map by composing with uniformly Christ-local bi-Lipschitz embeddings. See Lemma \ref{W-lem2}.
\begin{lemma}
\label{partition}
There exist functions ${ \varphi_{Q}}:X \rightarrow \R$ where $Q \in \M_{\Omega}$ with the following properties:
\begin{enumerate}
\item $0 \leq  {{\varphi}_{Q}} \leq 1 $,
\item ${{\varphi}_{Q}}|_{Q^{*}}= 1$,
\item  ${{\varphi}_{Q}}|_{X\setminus Q^{**}}= 0$,
\item ${{\varphi}_{Q}}$ is Lipschitz with constant $ \dfrac{C}{ \diam(Q)}$,
\item For all $p \in \Omega$, we have ${\varphi}_{Q}(p) \neq 0 $ for at most $N$ cubes $Q \in \M_{\Omega}$.
\end{enumerate}
Here, $C$ and $N$ denote uniformly fixed constants independent of the choice of element  $Q \in \M_{\Omega}$. They depend on the data of $X$, $\epsilon$, and the doubling constant of $\mu$. 
\end{lemma}
 \begin{proof}
 We define 
 \begin{equation*}
{\varphi_Q}(x) =\min \{1, \;\; \dfrac{\dist(x, X\setminus Q^{**})}{\dist(Q^*,\, X\setminus Q^{**})}\; \}.
\end{equation*} 
 Then, $(1),\,(2)$ and $(3)$ are obvious and $(5)$ follows from Proposition \ref{touch}. 
 To check $(4)$, note that
 \begin{equation*}
|\varphi_Q(p)-\varphi_Q(q)| \leq \dfrac{d(p,\,q)}{\dist(Q^*,\, X\setminus Q^{**})}.
\end{equation*}
Thus, it suffices to show that 
\begin{equation*}
\dist(Q^*,\, X\setminus Q^{**}) \geq c\,\diam(Q)
\end{equation*}
 To this end, let $x$ be a point in $Q^*$. We write $x \in R$ for some $R \in Q^*$ and choose $y \in S \in  X\setminus Q^{**}$. Then,
\begin{align*}
d(x,\, y) &\geq \dist(R,\, S) \\
&\geq \epsilon \min\{ \diam(R),\, \diam(S) \}\\
 &\geq C(L_1,\, A,\, \delta, \, \epsilon)\,\diam(Q).
\end{align*}
The last inequality is deduced from the comparability between $\diam(R)$ and  $\diam(Q)$ in case $\diam(S)\geq \diam(R)$. Otherwise,  $\diam(R) \geq \diam(S)$, we divide into two cases, either 
\begin{equation*}
(1)\;\diam(R) \geq \diam(S) \geq \dfrac{1}{2[\dfrac{4C_1\, A}{\delta}+1]}\, \diam(R) \;\;\;\;\;\text{or}\;\; \;\;\;(2) \;\diam(S) <\dfrac{1}{2[\dfrac{4C_1\, A}{\delta}+1]}\, \diam(R).
\end{equation*}
In the first case, we have obviously comparability between $\diam(S)$ and $\diam(R)$.  In the second case, we use the comparability condition of a Christ-Whitney decomposition. Then,
\begin{align*}
\dist(R,\, S) &\geq \dist(R, \,Y)- \dist(S, Y)-\diam(S)\\
&\geq \diam(R)-[\frac{4C_1\, A}{\delta}+1]\,\diam(S)\\
&\geq  \dfrac{1}{2}\, \diam(R)\\
&\geq C(L_1,\, A,\,\delta)\,\diam(Q).
\end{align*}
Therefore, the proof of $(4)$ is completed.
\end{proof}
\begin{rem}
We use the fact that $\varphi_Q =1$ on $Q^*$ and $\varphi_Q =0$ off $Q^{**}$ so that  the map $\widetilde{h}_Q =h_Q \cdot \varphi_Q$ defined in Subsection $3.3$ is bi-Lipschitz on $Q^{*}$ and supported on $Q^{**}$. These properties are needed in the proof of Lemma \ref{W-lem2}, see case $(3)$.
 \end{rem}
 
 \begin{defn}
 \label{measurelocal}
 Let $(X, d, \mu)$ be a uniformly perfect metric space supporting a  doubling measure and let $Y$ be a closed subset of $X$. We say that $\Omega =X\setminus Y$ admits  uniformly Christ-local bi-Lipschitz embeddings if there exist bi-Lipschitz embeddings of each $Q^{**}$ into a fixed Euclidean space with uniform bi-Lipschitz constant. 
\end{defn}

\subsection{Whitney Distance Map}
The following relative distance map plays a key role to construct a  co-Lipschitz map from a metric space into Euclidean space in Section 3.  We will break $\M_{\Omega}$  into two parts and construct co-Lipschitz maps on these parts (Definition \ref{W}) by using the Whitney distance map.
\begin{defn}
\label{W-distance}
The Whitney distance map $d_\W$ on $\M_{\Omega} \times \M_{\Omega}$ is
defined by
\begin{equation*}
d_{\W}(Q, R) = \dfrac{\dist(Q, \,R)}{\min(\diam(Q), \diam(R))}.
\end{equation*}
\end{defn}
\begin{rem}
The Whitney distance map $d_\W$ is not a metric. In fact, if $\overline{Q} \cap\overline{R} \neq \emptyset$, then $d_\W(Q,\, R) =0$. We observe that 
\begin{equation*}
d_\W(Q, \, R)\leq d_\W(Q,\, S) + d_\W(S,\, R)+1 
\end{equation*}
 if $\diam(S) \leq \min\{\diam(Q),\, \diam(R)\}$. 
 
 Throughout this paper, we will use the terminology Whitney distance ball of radius $\rho$ for the set of all elements in $M_{\Omega}$ such that Whitney distance to a fixed center cube in $\M_{\Omega}$ is less than $\rho$. We write $B_{\W}(Q,\, \rho)$ for the Whitney distance ball of radius $\rho$ with center $Q$.
\end{rem}

The next lemma allows us to construct a coloring map that gives different colors to Christ cubes within a given Whitney distance ball. 
\begin{lemma}
\label{number}
Each Whitney distance ball of radius $\rho$  contains a finite number of elements of the Christ-Whitney decomposition $\M_{\Omega}$. The number depends on the doubling constant of $\mu$ and  $\rho$.
\end{lemma}

\begin{proof}
We fix a Christ cube $Q \in\M_{\Omega}$ and we require to count the number of $R \in  \M_{\Omega}$ such that $d_\W(Q, R) <\rho$.
We have two cases either $(1)$ $ \diam (Q) < \diam(R)$ or  $(2)$ $ \diam (R) \leq \diam(Q)$.

Suppose $\diam(Q) <\diam (R)$. Then, we have
\begin{equation*}
\dist(R, \,Y) -\dist(Q,\, Y) <\dist(Q,\, R)+\diam(Q) <(\rho+1)\,\diam (Q).
\end{equation*}
 Since  $\dist(Q, \,Y) \leq   \dfrac{4C_1\,A}{\delta}\, \diam(Q)$,  we have an upper bound for  $\diam(R)$ in terms of $\diam(Q)$. That is, $\diam(R)< (\rho+1+\dfrac{4C_1\,A}{\delta})\, \diam(Q)$. 
 
 Similarly,   $\diam( R)$ has a lower bound in terms of the size of $Q$ in the case of $\diam(R) \leq \diam(Q)$:
\begin{equation*}
\diam(R) \geq   (\rho+1 +\dfrac{4C_1\,A}{\delta})^{-1} \,\diam(Q).
\end{equation*}
Therefore, the number of $R \in \M_{\Omega}$ in $B_{\W}(Q,\,\rho)$ is the sum of the  cardinality of the following sets:
\begin{equation}
\label{one}
\{R \in \M_{\Omega} \;|\; \diam(Q)< \diam(R) <(\rho+1+\frac{4C_1\,A}{\delta}) \,\diam(Q) \; \text{and}\; \dist(Q, R)<\rho\,\diam(Q) \}
\end{equation}
and
\begin{equation}
\label{two}
\{R\in \M_{\Omega} \;|\; (\rho+1 +\frac{4C_1\,A}{\delta})^{-1}\,\diam(Q) < \diam(R) \leq \diam (Q)\; \text{and}\; \dist(Q, R) <\rho\,\diam(R) \}
\end{equation}
Now we suppose that $p$ and $q$  are  centers of $C_1$- quasiballs $\widetilde{B_Q}$ and $\widetilde{B_R}$ which have approximately sizes of $Q$ and $R$. If $R$ is in either the set  \eqref{one} or the set \eqref{two}, then we find that
\begin{equation}
\label{dist}
d(p, \,q) \leq \diam(Q)+\dist(Q,\, R)+\diam(R) <(2\rho+1+\frac{4C_1\,A}{\delta})\,\diam(Q).
\end{equation}
Thus, the number of $R \in \M_{\Omega}$ in $B_{\W}(Q,\, \rho)$  is at most twice of the number of centers $q$  satisfying \eqref{dist}. In other words, we can count the number of $R$'s  in \eqref{one} and \eqref{two} by counting the number of centers of $C_1$-quasiballs $\widetilde{B_R}$. By the doubling condition, the ball centered at $p$ with radius  $(2\rho+1+\dfrac{4C_1\,A}{\delta})\,\diam(Q)$ can be covered by finite number of $C_1$-quasiballs centered at such $q$. Finally, the comparability of the  size of $R$ and that of the ball centered at $q$ concludes  Lemma \ref{number}.
\end{proof}

We write the number of Christ cubes within Whitney distance ball  of radius $\rho$ as $m =m(\rho, \, D)$ in terms of  $\rho$ and  the doubling constant $D$ of $\mu$.
\begin{lemma}
\label{color}
There exists a coloring map
\begin{equation*}
K : \M_{\Omega} \longrightarrow \{1,2,3,\ldots, M \}\; \text{for some}\; M \geq m(m-1)
\end{equation*}
such that any two boxes within Whitney distance ball of radius $\rho$ have different colors. In other words,
if  $R',\; R''$ have  $d_{\W}(R', R'') <\rho$, then $K(R') \neq K(R'')$.
\end{lemma}

\begin{proof}
We apply Zorn's lemma. Let us consider the partially ordered set $(\mathcal{P}, \leqslant)$ where
$\mathcal{P}$ is the collection of maps $k$ defined from $\mathcal{S} \subset \M_{\Omega}$ to  $ \{1,2, \ldots ,M\}$ so that $K(R) \neq K(R')$ for all  $R,\, R' \in  \mathcal{S}$ whose Whitney distance is less than $\rho$. The inequality $(k, \mathcal{S}) \leqslant  (k^{\prime},\mathcal{ S^{\prime}})$ means $ k^{\prime}$ is a extension of $k$ $(\mathcal{S}\subset \mathcal{S^{\prime}} \in \mathcal{P} \; \text{and} \; k^{\prime} |_{\mathcal{S}} = k )$. 

By Zorn's lemma, there exists a maximal element $ \widehat{k} $. If the domain of $ \widehat{k}$ is $\M_{\Omega}$, then we can set $K=\widehat{k}$. Otherwise, take $Q^{\prime} \in \M_{\Omega} \setminus \text{domain}( \widehat{k} )$. We now want to give a color to $Q'$. The color of $Q'$ should differ from any color already assigned to any $R$ where $d_{\W}(Q',\, R)<\rho$ and also differ from any color already assigned to any $S$ where $d_{\W}(S,\, R) <\rho$ and $d_{\W}(Q',\, R)<\rho$. We observe that the number of such $R$ is at most $m-1$ and the number of $S$ for given $R$ is at most $m$. Thus, the total number of colors seen is at most $m(m-1)$. Since $M\geq m(m-1)$, there is a color remaining which can be assigned to $Q'$; this contradicts maximality of $\widehat{k}$.
\end{proof}
\section{Bi-Lipschitz embeddable metric spaces}
Now we are ready to state the main theorem. It asserts that in a uniformly perfect complete metric space supporting a doubling measure, the local information of uniformly Christ-local bi-Lipschitz embeddability (Definition \ref{measurelocal}) can be turned into global information of bi-Lipschitz embeddability. 
\begin{thm}
\label{main}
A uniformly perfect complete metric space $(X, d)$ admits a bi-Lipschitz embedding into some Euclidean space if and only if the following conditions hold:
\begin{enumerate}
\item it supports a doubling  measure $\mu$,
\item there exists a closed subset $Y$ of $X$ which admits a bi-Lipschitz embedding into some $\R^{M_1}$, 
\item $\Omega = X\setminus Y$ admits uniformly Christ-local bi-Lipschitz embeddings into some $\R^{M_2}$. 
\end{enumerate}
The bi-Lipschitz constant and dimension of receiving Euclidean space depend on the data of the metric space $X$, the doubling constant of $\mu$,  $M_1$, $ M_2$, and the bi-Lipschitz constants in conditions $(2)$ and $(3)$.
\end{thm}

\noindent{\bf Outline of Proof}\\
Suppose that we have a $L$-bi-Lipschitz embedding $f$ from $(X, d)$ into $\R^n$ for some $n$.  Euclidean space is a doubling metric space and the doubling condition is bi-Lipschitz invariant. Hence, $(X, \,d)$ is a complete doubling metric space. Thus, there exists a doubling measure $\mu$ (\cite{sak}, \cite{Wu}).  The second condition is trivial, setting $Y= X$. The third condition is trivial since $\Omega=\emptyset$.

The content of the theorem is the other implication: a uniformly perfect complete space satisfying $(1)$, $(2)$, and $(3)$ embeds bi-Lipschitzly in some $\R^n$ for some $n$.
We will use Proposition \ref{mainlem} to complete the main theorem. Since the full measure set $\M_{\Omega} \cup Y$ is dense in $X$  and  the constructed map in Proposition \ref{mainlem} is uniformly continuous, Theorem \ref{main} follows immediately. Therefore, we will focus on proving Proposition \ref{mainlem} in Subsection \ref{W-large}, Subsection \ref{W-local}, and Subsection \ref{Glo}.

\begin{prop}
\label{mainlem}
Let $(X, d, \mu)$ be a $A$-uniformly perfect, complete, doubling metric space and let $Y$ be a closed subset of $X$. Then, 
the  full measure set $M_{\Omega} \cup Y$ admits a bi-Lipschitz embedding into some Euclidean space if the followings are satisfied:
\begin{enumerate}
\item $Y$ admits a bi-Lipschitz embedding into some $\R^{M_1}$,
\item  $\Omega = X\setminus Y$ admits uniformly Christ-local  bi-Lipschitz embeddings into some $\R^{M_2}$. 
\end{enumerate}
 The bi-Lipschitz constant and dimension of receiving Euclidean space depend on the data of metric space $X$, the doubling constant,  $M_1$, $ M_2$, and the bi-Lipschitz constants in conditions $(1)$ and $(2)$. 
\end{prop}
We  briefly outline the proof of Proposition \ref{mainlem}. We first extend a (bi)-Lipschitz map $f$ on $Y$ to a global Lipschitz map $g$ on $X$, using McShane's  extension theorem (see Theorem \ref{McShane} and Corollary \ref{McShane2}). We then suppose that $f$ is a $L_1$-bi-Lipschitz embedding from $Y$ into $\mathbb R^{M_1}$. From McShane's  theorem, we have a $\sqrt{M_1} L_1$-Lipschitz extension map
\begin{equation*}
g: X \longrightarrow \mathbb R^{M_1}\; \;\text{such that}\;\; g|_Y=f.
\end{equation*}
From now on we fix such $L_1$  and  $M_1$ is chosen sufficiently large relative to other data $C_1 ,\, A,$ and $\delta$. The precise choice of $M_1$ will be made in connection with the estimate in \eqref{$M_1$}.

In general, the map $g$ is not globally co-Lipschitz on a full measure set $\M_{\Omega}$ of $\Omega$. Therefore, we next shall construct a co-Lipschitz map using a local and large scale argument in the sense of Whitney distance on a Christ-Whitney decomposition (see Definition \ref{W-distance} and Definition \ref{W}).
\begin{defn}
\label{W}
Let $Q$ be any fixed cube in $ \M_{\Omega}$.
We say $f :\M_{\Omega} \rightarrow \R^n$  is $\W$-local co-Lipschitz if  it is co-Lipschitz for any two points $p \in Q$, $q\in R$ where $R $ is in $B_{\W}(Q,\, 16M_1{L_1}^2)$. We say $f$ is $\W$-large scale co-Lipschitz if it is co-Lipschitz for any two points $p \in Q$ and $q \in R$ where $R$ is not in $B_{\W}(Q,\, 16M_1{L_1}^2)$.
\end{defn}

In Subsection \ref{W-large}, we will construct a $\W$-large scale co-Lipschitz map and global Lipschitz map on $\M_{\Omega}$. To this end, we will break the complement of an arbitrary  Whitney distance ball of radius $16 M_1{L_1}^2$ into two parts using relative distance in terms of the distance between two cubes and their maximum diameter.  We shall see that McShane's extension map $g$ and distance map from $Y$, $\dist(\cdot, \,Y)$, which are global Lipschitz maps, are $\W$-large scale co-Lipschitz on these two parts respectively.
 
In Subection \ref{W-local}, we will construct a $\W$-local co-Lipschitz map on $\M_{\Omega}$ via putting together all local patches of bi-Lipschitz embeddings. We will assign different colors to elements in a Christ-Whitney decomposition within arbitrary Whitney distance ball of radius $16 M_1{L_1}^2$.
  
Finally, in Subection \ref{Glo}, we will construct a global bi-Lipschitz embedding on the full measure set $\M_{\Omega}\cup Y$ of $X$ completing the proof of Lemma \ref{mainlem}.
\begin{figure}[h]
 \centering\subfloat[$\W$-local co-Lipschitz]
  {\includegraphics[width = 6cm]{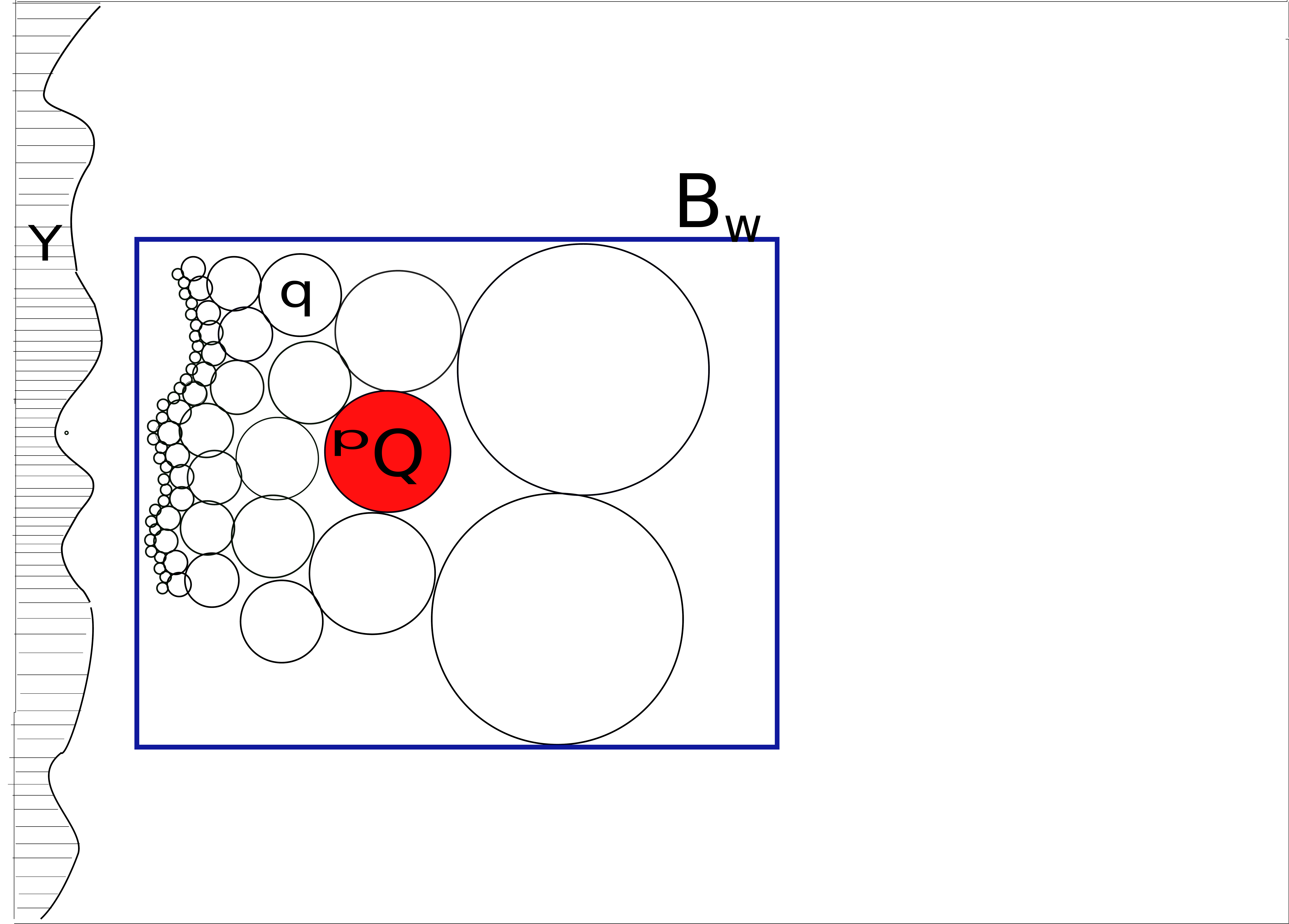}}\hspace*{3cm}
 \subfloat[$\W$-large scale co-Lipschitz]
  {\includegraphics[width = 6cm]{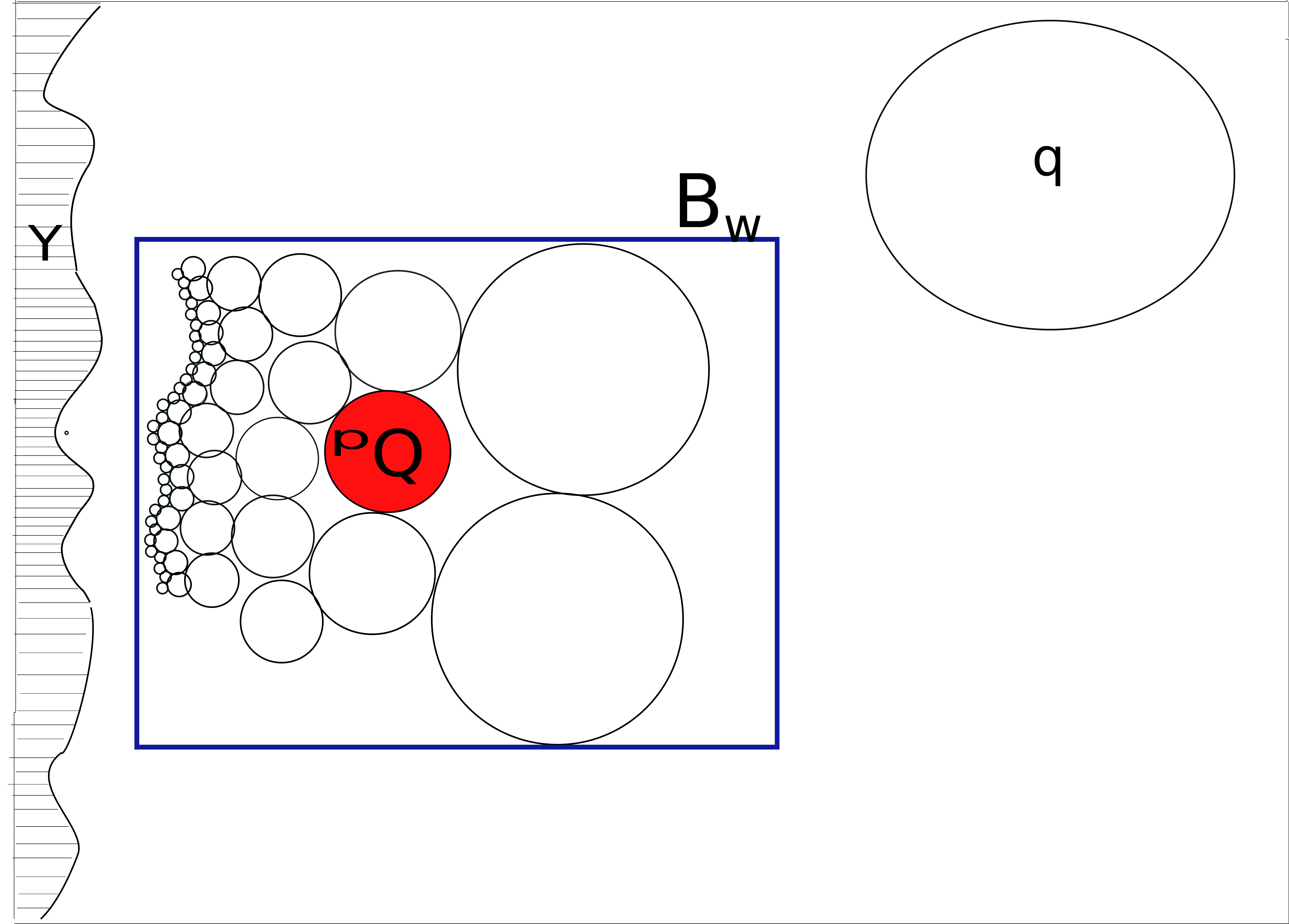}}
 \caption{ Let the square be the Whitney distance ball of radius $16{M_1}{L_1}^2$ centered at $Q$.
 $\W$-local co-Lipschitz means $|f(p)-f(q)|\gtrsim d(p,\, q)$ for any $p\in Q$ and $q\in R$ where $d_{\W}(Q,\, R)<16 {M_1}{L_1}^2$. $\W$-large scale co-Lipschitz means $|f(p)-f(q)|\gtrsim d(p,\, q)$ for $p\in Q$ and $q\in R$ with $d_{\W}(Q,\, R)\geq16 {M_1}{L_1}^2$. }
 \label{fig2}
\end{figure}

\subsection{$\W$-Large Scale Co-Lipschitz and Global Lipschitz Map on $\M_{\Omega}$ }
\label{W-large} 
We construct a $\W$-large scale co-Lipschitz and global Lipschitz map on a full measure set $\M_{\Omega} \subset \Omega$. Roughly speaking, McShane's extension map guarantees a $\W$-large scale co-Lipschitz bound for points $p,\, q$ in $\M_{\Omega}$ whose distance is big enough with respect to the maximum diameter of cubes containing them. Whenever $p \in Q$ and $q \in R$ have the property that $\dist(Q, \,R)$ exceeds their maximum diameter by a fixed constant, we consider points $z ,\,z'$ in $Y$ which realize distances to $p,\,q$ respectively. Then,  
$|g(p)-g(z)|$ and $|g(q)-g(z')|$ are approximately greater than the maximum diameter and we can conclude co-Lipschitz from the triangle inequality. Furthermore, when the distance between two points is small enough with respect to the maximum diameter,  $|d( p,\, Y) -d(q,\, Y)|$ is approximately greater than the maximum diameter (see Figure \ref{fig3}).

\begin{lemma}
\label{W-lem1}
Let $Q$ be any fixed cube in $\M_{\Omega}$.
For any two points $p \in Q$ and $q \in R$, where $d_{\W}(Q,\, R)\geq 16{M_1}{L_1}^2$, the  McShane extension map $g$ and $\dist(\cdot,\, Y)$ guarantee W-large scale co-Lipschitz bounds. More precisely, 
\begin{enumerate}
 \item If $\dfrac {\dist(Q, \,R)}{\max(\diam(Q),\diam(R))} \geq \dfrac{8{M_1}{L_1}^2}{1+\frac{4\,C_1\,A}{\delta}}$, then $|g(p)-g(q)| \geq C(L_1, \,M_1) \, d(p,\, q)$.
 \item If $\dfrac{\dist(Q, \,R)} {\max(\diam(Q),\diam(R))}  \leq \dfrac{8{M_1}{L_1}^2}{1+\frac{4\,C_1\,A}{\delta}}$, then $|\dist(p, \,Y)-\dist(q, \,Y)| \geq C(L_1,\, M_1) \, d(p,\, q)$.
\end{enumerate}
\end{lemma}

\begin{figure}[h]
 \centering
 \subfloat[$\dfrac {\dist(Q, \,R)}{\max(\diam(Q),\diam(R))} \geq  \dfrac{8{M_1}{L_1}^2}{1+\frac{4\,C_1\,A}{\delta}}$]
  {\includegraphics[width = 6cm]{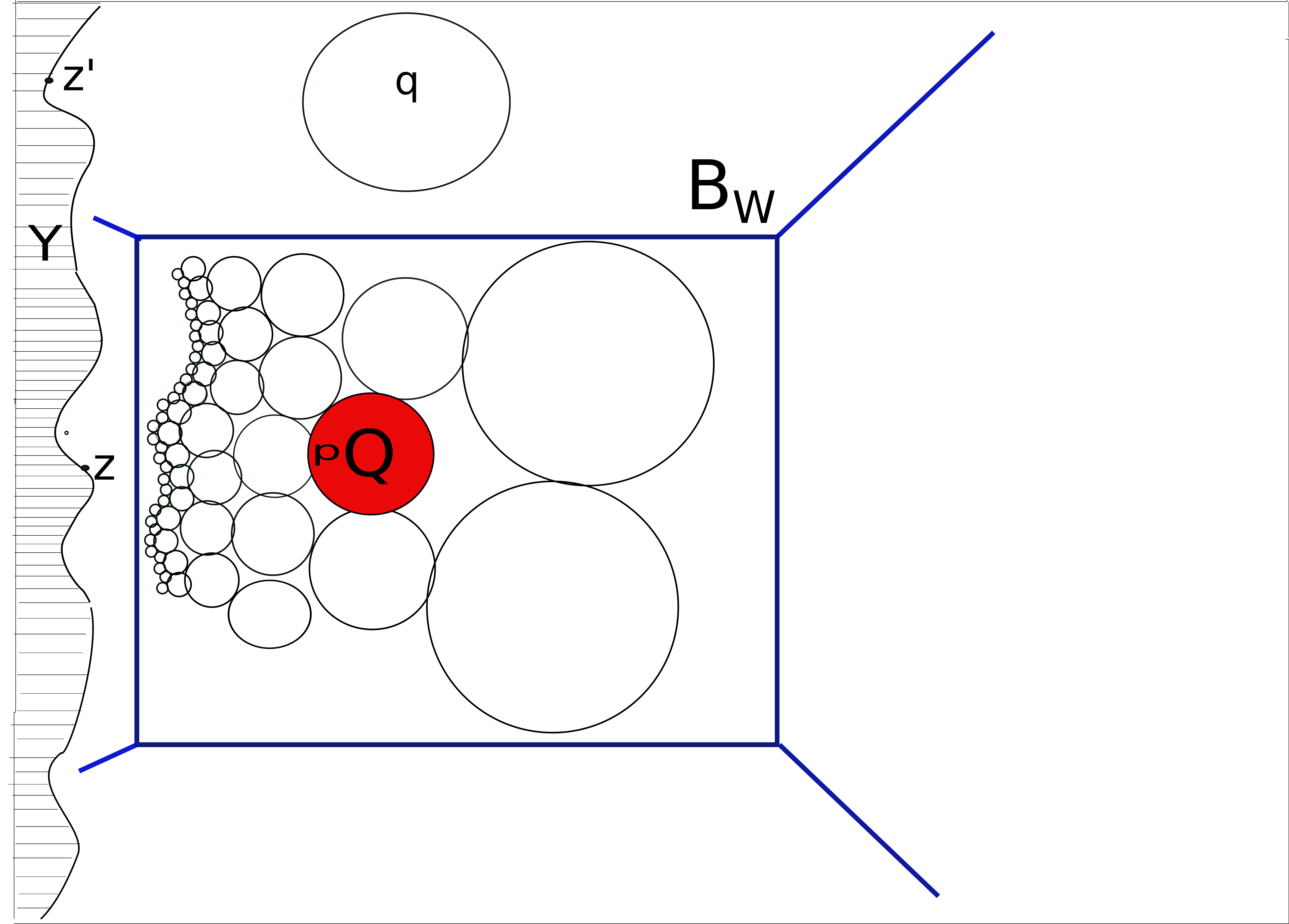}} \hspace*{3cm}
 \subfloat[$\dfrac{\dist(Q, \,R)} {\max(\diam(Q),\diam(R))}  \leq  \dfrac{8{M_1}{L_1}^2}{1+\frac{4\,C_1\,A}{\delta}}$]
  {\includegraphics[width = 6cm]{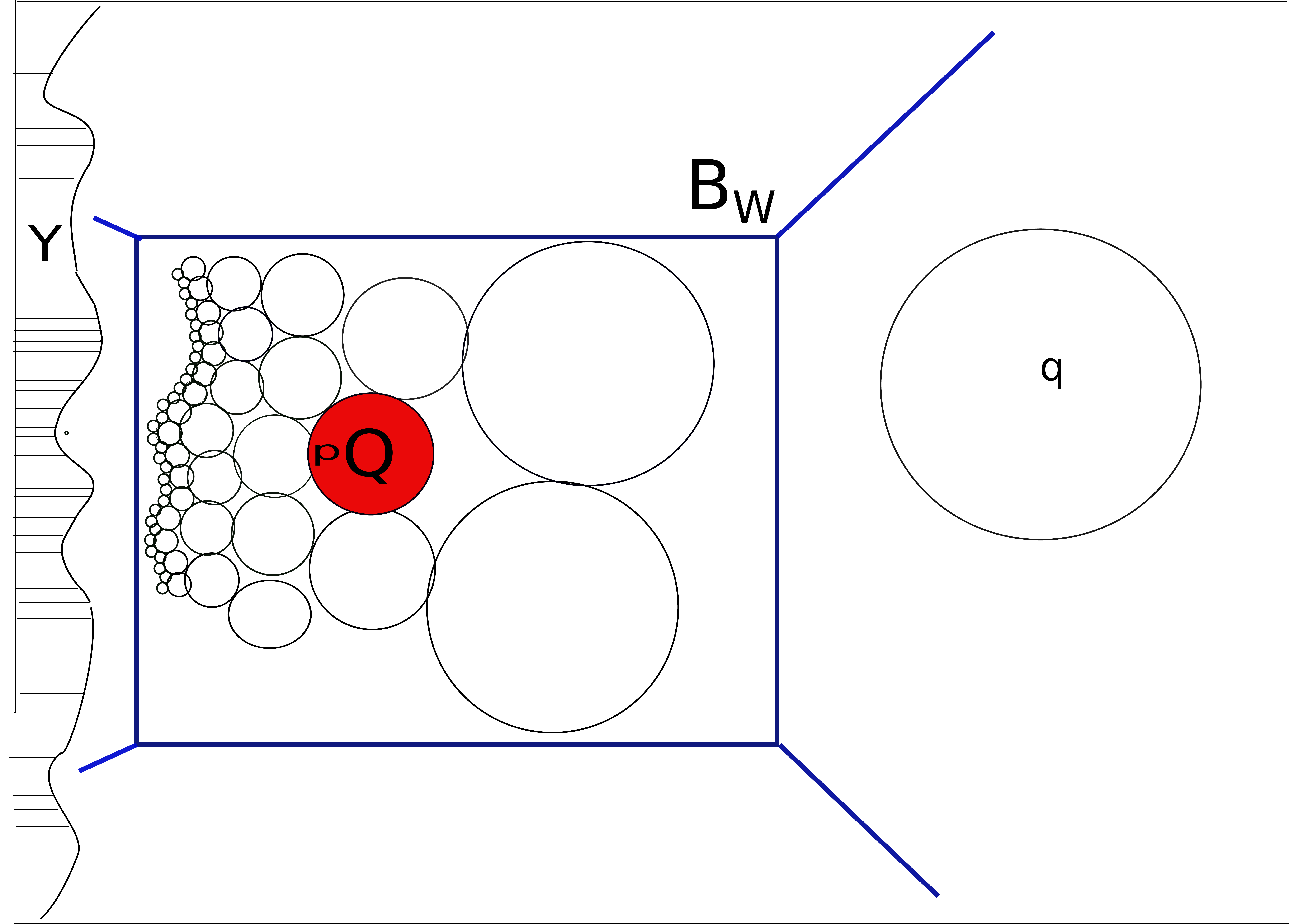}}
 \caption{$g$ and $\dist(\cdot,\, Y)$ guarantee $\W$-large scale co-Lipschitz bounds. }
 \label{fig3}
\end{figure}
\begin{proof}
We may assume that $\diam(R) \geq \diam(Q)$ without loss of generality. We  choose $z, \,z^{\prime} \in Y$ such that $ \dist(Y, \,Q) =\dist(z,\,Q)$ and $ \dist(Y, \,R) =\dist(z^{\prime}, \,R)$. We claim that $z\neq z'$. In fact, $d(z,\,z')\geq \frac{1}{2}\,d(p,\,q)$. To conclude the claim, we suppose that $d(z,\,z')<\frac{1}{2}\,d(p,\,q)$. Then,
\begin{equation*}
d(p,\,q) \leq d(p,\,z)+d(z,\, z')+d(z',\, q).
\end{equation*}
Thus, we have 
\begin{align*}
d(p,\,q) &\leq 2\,[d(p,\,z)+d(z',\, q)]\\
            &\leq 2\,[\dist(z,\, Q)+\diam(Q)+\dist(z',\, R)+\diam(R)]\\
            &\leq 2\,[\dist(Y,\, Q)+\diam(Q)+\dist(Y ,\, R)+\diam(R)]\\
            & \leq 2\,(\dfrac{4C_1A}{\delta}+1)[\diam(Q)+\diam(R)]\\
            &\leq 2\,(\dfrac{4C_1A}{\delta}+1)(\dfrac{1+\frac{4C_1A}{\delta}  }{8M_1{L_1}^2 }+\frac{1}{16M_1{L_1}^2} )\,\dist(Q,\, R)\\
            &\leq \dfrac{ (1+\frac{4C_1\, A}{\delta} )( 3+\frac{4C_1\, A}{\delta} ) }{ 8M_1{L_1}^2}\,d(p,\,q).
\end{align*}
This is a contradiction  provided $M_1$ is selected sufficiently large relative to $C_1,\, A$, and $\delta$.
Now, 
\begin{align*}
|g(p)-g(q)|&\geq |f(z)-f(z^{\prime})|-|f(z)-g(p)|-|f(z^{\prime})-g(q)|\\
                  & \geq \dfrac{1}{L_1}\,d(z,\, z')-C\,d(z,\,p)-C\,d(z',\, q).                                             
\end{align*}
where $C=\sqrt{M_1}L_1$ from McShane's theorem.
We have
\begin{align*}
d(p, z) &\leq (\frac{4C_1\,A}{\delta} +1)\,\diam(Q) \leq \dfrac{(\frac{4C_1\,A}{\delta}+1)}{16{M_1}{L_1}^2}\,\dist(Q,\, R) \\
           &\leq \dfrac{(\frac{4C_1\,A}{\delta}+1)}{16{M_1}{L_1}^2}\,d(p,\,q).
 \end{align*}
Similarly, we have
\begin{equation*}
|g(q)-g(z^{\prime})| \leq L_1 \,d(z^{\prime}, \,q) \leq \dfrac{(\frac{4C_1\,A}{\delta}+1)^2}{8{M_1}{L_1}^2}\, d(p,\,q). 
\end{equation*} 
In conclusion, 
\begin{equation}
\begin{split}
\label{$M_1$}
|g(p)-g(q)| & \geq [\dfrac{1}{2L_1} -2C \dfrac{(\frac{4C_1\,A}{\delta}+1)^2}{8{M_1}{L_1}^2}]\,d(p,\,q)\\
                   &\geq \dfrac{1}{2L_1}[1-\dfrac{(\frac{4C_1\,A}{\delta}+1)^2}{2\sqrt{M_1}}]\,d(p,\,q)\\
                   &\geq \dfrac{1}{4L_1}\,d(p,\,q)
\end{split}
\end{equation}
since we can choose $M_1$ sufficiently large. This completes the proof of the first case.

In second case, we have
\begin{equation*}
16M_1{L_1}^2\diam(Q) \leq  \dist(Q,\, R) \leq \dfrac{8{M_1}{L_1}^2}{1+\frac{4\,C_1\,A}{\delta}}\,\diam(R).
\end{equation*}
Therefore, $2\,(1+\dfrac{4\,C_1\, A}{\delta})\,\diam(Q) \leq \diam(R)$.
We now have
\begin{align*}
|\dist(p\, Y)-\dist(q\, Y)| &\geq \dist(q, \,Y)-\dist(p,\,Y)\\
                                     &\geq \dist(R,\,Y)-\dist(Q,\, Y)-\diam(Q)\\
                                     &\geq \diam(R)-(1+\dfrac{4\,C_1\,A}{\delta})\,\diam(Q)\\
                                     &\geq \frac{1}{2}\,\diam(R)
\end{align*}
while $d(p,\,q) \leq \diam(Q)+\dist(Q,\,R)+\diam(R) \lesssim \diam(R)$. Thus, we proved the second case.
\end{proof}

\subsection{$\W$-Local Co-Lipschitz and Global Lipschitz Map on $\M_{\Omega}$ }
\label{W-local}
 We next construct a $\W$-local co-Lipschitz and global Lipschitz map on a full measure set $\M_{\Omega} \subset \Omega$ into some  Euclidean space. In general, $M_1+1$, the dimension of the target space  of $g (\cdot)\times \dist(\cdot, \, Y)$ is not large enough to construct a co-Lipschitz map. Hence, we will use a coloring map that gives additional dimension of the Euclidean space (see Lemma \ref{color}).

Suppose that  $h_{Q}$'s are $L_2$-bi-Lipschitz embeddings of $Q^{**}$ for each $Q \in \M_{\Omega}$ into $\mathbb R^{M_2}$ with uniformly determined  $L_2$ and $M_2$.  Now we consider the map
\begin{equation*}
 \widetilde{h}_{Q}=h _{Q}\cdot \varphi_{Q}: X \longrightarrow \R^{M_2};
 \end{equation*} 
 it is bi-Lipschitz on $Q^*$, Lipschitz on $X$, and supported on $Q^{**}$. We recall that $\{\varphi_Q\}$ is a family of Lipschitz cutoff functions as in Lemma \ref{partition}.
 Then, we may assume that for some $c$
\begin{equation*}
\widetilde{h}_Q(Q^{*}) \subset B(0, c\,L_2\diam(Q)) \setminus B(0, \dfrac{1}{c\,L_2} \diam(Q))
\end{equation*}
 because we can postcompose with an isometric translation map of $\R^{M_2}$  if necessary.
Next, we will put together all patches to make a $\W$-local co-Lipschitz map by assigning different colors to each element in $M_{\Omega}$. We will denote $\{e_1,\; e_2,\; \ldots,\; e_M \}$ by an orthonormal basis for $\mathbb R^{M}$.
\begin{lemma}
\label{W-lem2}
The following map $H$ from $\M_{\Omega}$ into $(\R^{M_2})^M$ given by
\begin{equation}
H(p) = \sum_{Q \in \M_{\Omega}} \widetilde{h}_{Q}(p) \otimes e_{K(Q)},
\end{equation}
is a global Lipschitz and W-local co-Lipschitz map. The ($\W$-local) bi-Lipschitz constant depends on $L_1,\;L_2\; \text{and}\; M_1$.  That is, 
\begin{equation*}
 |H(p)-H(q)| \geq  C(L_1, \,L_2,\, M_1 ) \, d(p, q) 
\end{equation*}
 for  any points $p$ in any fixed $Q$ and $q$ in $R $ where $d_\W(Q,R)< 16{M_1}{L_1}^2$.
\end{lemma}

\begin{proof}
Since $\widetilde h_{Q}$ is bi-Lipschitz on $Q^*$ with the uniform bi-Lipschitz constant $L_2$, Lipschitz on $X$, and supported on $Q^{**}$, the map $H$ is a  finite sum of Lipschitz maps from Proposition \ref{touch}. Thus, it is  Lipschitz on ${\Omega}$. Now, we will show that $H$ is a $\W$-local co-Lipschitz map according to positions of two points $p$ and $q$ on $\M_{\Omega}$. There are three cases.

$(1)$ If $p,\; q \in Q^*$, then $\widetilde h _{Q}$ is bi-Lipschitz on $Q^*$ and $Q$ is the element in $\M_{\Omega}$ that shares the same color at $p$ and $q$. Therefore, we find that
\begin{align*}
| H(p)-H(q)|& \geq | \widetilde h_{Q}(p)- \widetilde h_{Q}(q)|\\
                      &=|h_{Q}(p)- h_{Q}(q)|\;\;\;  \; \text{since}\;{\varphi_{Q}}|_{Q^*} = 1\\
                      &\geq \frac{1}{L_2}\,d(p,\, q)
\end{align*}
since $h_{Q}$ is $L_2$-bi-Lipschitz.

$(2)$ If $p \in Q,\; q\notin Q^{**}$, then $\widetilde h_{Q}(q) = 0$. Thus, we have
\begin{align*}
 | H(p)-H(q)| &\geq | \widetilde h_{Q}(p)- \widetilde h_{Q}(q)|=|\widetilde h_{Q}(p)|\\
              &\geq \dfrac{1}{c\,L_2}\,\diam(Q).
\end{align*}              
On the other hand, we observe that
\begin{align*}
d(p,\,q) &\leq \diam(Q)+\dist(Q,\, R) +\diam(R)\\
            &\leq \diam(Q) +\dist(Q,\, R)+\dist(R,\,Y)\\
            &\leq 2\diam(Q)+2\dist(Q,\,R)+\dist(Q,\,Y).
\end{align*}              
Since $\dist(Q,\,Y)\leq \dfrac{4\,C_1\, A}{\delta}\,\diam(Q)$ and 
\begin{equation*}
\dist(Q,\,R) \leq 16M_1{L_1}^2\min\{\diam(Q),\, \diam(R)\}\leq16M_1{L_1}^2\diam(Q),
\end{equation*}
 we conclude
\begin{equation*}
d(p,\,q) \lesssim \diam(Q)
\end{equation*}
 and so $|H(p)-H(q)|\gtrsim d(p,\,q)$ as desired.
 
$(3)$ If $p \in Q,\; q\in Q^{**}$, then there is  $R \in Q^*$ so that $p,\;q \in R^*$ and $\widetilde h_{R}$ is bi-Lipschitz on $R^*$. Therefore, we conclude the following from the first case:
\begin{equation*}
 |H(p)-H(q)| \geq |\widetilde h_{R}(p)-\widetilde h_{R}(q)| \geq \frac{1}{L_2}\,d(p,\,q).
\end{equation*}
\end{proof}

\subsection{Global Bi-Lipschitz Embedding on a Full Measure Set  $ \M_{\Omega} \cup Y$ }
\label{Glo}
Finally, we are ready to construct a global bi-Lipschitz embedding on a full measure set of $X$. We define the map $F$ from $\M_{\Omega}\cup Y$ into $\mathbb R^{M_1} \times\mathbb (R^{M_2})^M \times \mathbb {R}$ as follows:
\begin{equation}
\label{global}
F(p)=
\begin{cases}
g(p) \times H(p) \times \dist(p, Y),& \text\;{for}\; p \in \M_{\Omega}\,;\\
f(p)\times \{0\}\times \{0\} , & \text{for}\; p \in Y.
\end{cases}
\end{equation}
\textit{}
Then $F$ is Lipschitz on a full measure set $\M_{\Omega} \subset \Omega$ because $g$ and  $\dist(\cdot, \,Y)$ are Lipschitz on $X$ and $H$ is a finite sum of Lipschitz maps on $\M_{\Omega}$. Moreover, when we define $H(q)=0$ for $q \in Y$, then for  every $p \in \M_{\Omega}$ and any $q \in Y$, we arrive at
\begin{align*}
|H(p)-H(q)|=|H(p)| =& |\sum_{Q \in \M_{\Omega}} \widetilde{h}_{Q}(p) \otimes e_{K(Q)}|\\
                                       &\leq N\,L_2\,\diam(Q)\\
                                       &\leq N\,L_2\,\dist(Q, Y)\\
                                      &\leq N\,L_2\,d(p,q)
\end{align*}
 We have shown that $F$ is co-Lipschitz on  $\M_{\Omega}$ by Lemma \ref{W-lem1} and Lemma \ref{W-lem2} and $F|_Y = f$ is co-Lipschitz. Finally, we have a bi-Lipschitz embedding $F$ from  a full measure set $\M_{\Omega}\cup Y $ of  $X$  into $\mathbb R^{M_1} \times\mathbb (R^{M_2})^M \times \mathbb {R}$. The bi-Lipschitz constant depends on the data of metric space $X$, the doubling constant of $\mu$,  $M_1$, $ M_2$, $L_1$ and $L_2$.
 Therefore, Proposition \ref{mainlem} is proved.
\section{Applications}\label{Grushin}
We recall the Rademacher-type theorems of Pansu and Cheeger. Then, we discuss their applications to the problem of bi-Lipschitz nonembedding. In contrast,  as an application of Theorem \ref{main} we will prove that the Grushin plane equipped with Carnot-Carath\'{e}odory distance embeds bi-Lipschitzly into Euclidean space.  Thus, we obtain the first example of a sub-Riemannian manifold admitting a bi-Lipschitz embedding.

We recall definitions of the Heisenberg group and the Grushin plane.
\begin{defn}
\label{Heisenberg}
The Heisenberg group $\mathbb H$ is $\mathbb R^3$ with horizontal distribution spanned by two vectors
\begin{equation*}
X_1=\dfrac{\partial}{\partial x}-\dfrac{y}{2}\dfrac{\partial}{\partial z }\;\;\;\;\text{  and }\;\;\;\; X_2=\dfrac{\partial}{\partial y}+\dfrac{x}{2}\dfrac{\partial}{\partial z}.
\end{equation*}
It is the first non trivial example of step 2 Carnot group and it has dilations
\begin{equation}
\label{Heisenbergdilation}
 \delta_{\lambda}(x,\, y,\, z) =(\lambda x,\, \lambda y,\, {\lambda}^2z).
 \end{equation}
\end{defn}
\begin{defn}
\label{G}
The Grushin plane $\mathbb G$ is $\R^2$ with horizontal distribution spanned by 
\begin{equation*}
X_1=\dfrac{\partial}{\partial x} \;\;\; \text{and} \;\;\; X_2=x\,\dfrac{\partial}{\partial y}.
\end{equation*}
\end{defn}
We next define  dilations $\delta_{\lambda}$ on $\mathbb G$   by
\begin{equation}
\label{dilation}
\delta_{\lambda}(x,\, y)=(\lambda x, \, \lambda^{2}y)
\end{equation}
whenever $p=(x,\, y) \in \mathbb{G}$ and $\lambda >0$. Then,  $X_1$ and $X_2$ are homogeneous of degree one with respect to the dilations. Hence, the Carnot-Carath\'{e}odory  distance satisfies 
\begin{equation}
d_{cc}\,(\delta_{\lambda}(p,\, q))=\lambda\, d_{cc}(p, q).
\end{equation}
 for all $p,\, q \in \mathbb {GS}$.  See \eqref{metric}  for definition of Carnot-Carath\'{e}odory distance. The points on the line $x=0$ are singular, while the other points in the plane are regular.  We write $\mathbb A$ for the set of singular points. For more information about the Grushin plane, see \cite{Bell}. 
 \begin{rem}
\label{Grushinrem}
The Grushin plane with Carnot-Carath\'{e}odory distance is a globally doubling measure space and satisfies globally Poincar\'{e} inequality with respect to Lebesgue measure .
\end{rem}
 The Grushin plane $\mathbb{G}$ with Lebesgue measure is a locally doubling metric measure space satisfying  locally $p$-Poincar\'{e} inequality for any $p\geq 1$ (\cite{sobolev}, \cite{Jerison}). We fix a compact $K$ which contains a neighborhood of the origin and $r_0 >0$.  For any $p\in \mathbb G$ and any $r>0$, we choose $\lambda >0$ so that $\delta_{\lambda} (B(p,\,2r)) = B(\delta_\lambda(p),\,2\lambda\, r)$ is contained in $K$ and $\lambda \,r \leq r_0$. Then the doubling condition holds for $\delta_{\lambda}(B(p,\,r)) = B(\delta_{\lambda}(p),\,\lambda r)$ and $\delta_{\lambda}(B(p ,\, 2r)) = B(\delta_{\lambda}(p),\,2\lambda r)$. Since $\mu (\delta_{\lambda}(E)) ={ \lambda}^3 \mu(E)$  for any set $E \subset \mathbb G$ we conclude the doubling condition for $B(p,\,r)$. A similar argument applies to the Poincar\'{e} inequality.

\subsection{Bi-Lipschitz Nonembedding Theorems}
In Euclidean space, Rademacher's theorem states that a Lipschitz function is differentiable almost everywhere  and the derivative is linear. We shall state theorems of Pansu and Cheeger which are analogues of Rademacher's theorem in some sense. These theorems can be applied to get nonembeddability of some metric spaces into Euclidean space. 
\begin{thm}[ Pansu \cite{Pansu}]
\label{Pansu}
Let $(M,\,  \bullet )$ and $(N ,\,  \star  )$  be Carnot groups. Every Lipschitz mapping $f$ between open sets in $M$ and $N$ is differentiable almost everywhere. Moreover, the differential 
\begin{equation*}
df_y(x)=\lim_{t\rightarrow0}\delta_{t^{-1}} [{f(y)}^{-1}\star f(y \bullet \delta_t(x))]
\end{equation*}
 is a Lie group homomorphism almost everywhere.
\end{thm}
Here $(\delta_t)$ denotes the family of dilations in $M$ or $N$. See \eqref{Heisenbergdilation} for the case of the Heisenberg group.

Semmes \cite{Semmes1} observed that Theorem \ref{Pansu} implies that nonabelian Carnot groups $M$ can not be embedded bi-Lipschitzly in Euclidean space.  If $M$ had a bi-Lipschitz embedding $f$ into some Euclidean space $\R^n$, then $f$ must be differentiable in the sense of Pansu and its differential should be an isomorphism. This gives a contradiction because it has nontrivial kernel. Hence $M$ cannot be bi-Lipschitz embeddable. In particular,  the Heisenberg group does not admit a bi-Lipschitz embedding into Euclidean space.

Rademacher's theorem states that infinitesimal  behavior of any Lipschitz functions on $\R^n$ is approximated at almost every point  by some linear function; that is, a linear combination of the coordinate functions. Cheeger proved a remarkable extension of Rademacher's theorem in doubling metric measure spaces supporting a $p$-Poincar\'{e} inequality. He constructed coordinate charts that span the differentials of Lipschitz functions. Moreover, his work gives a way to get nonembeddability results by using a purely geometric and analytic method.
\begin{thm}[Cheeger  \cite{Cheeger1}]
\label{Cheeger}
 If $(X, d, \mu)$ is a doubling metric measure space supporting  a $p$-Poincar\'{e} inequality for some $p \geq 1$, then $(X, d, \mu)$  has a strong measurable differentiable structure, i.e. a countable collection of coordinate patches $\{(X_{\alpha},\pi_{\alpha}) \}$ that satisfy the following conditions:
\begin{enumerate}
 \item Each $X_{\alpha}$ is a measurable subset of $X$ with positive measure and the union of  the $X_{\alpha}$'s has full measure in $X$.
 \item Each $\pi_{\alpha}$ is a $N(\alpha)$-tuple of Lipschitz functions, for some $N(\alpha) \in \mathbb{N}$, where $N(\alpha)$ is bounded from above independently of $\alpha$.
\item Given a Lipschitz function $f:X \longrightarrow\mathbb{R}$, there exists an $L^{\infty}$ function $df^{\alpha}: X_{\alpha} \longrightarrow \mathbb R^{N(\alpha)}$ so that 
\begin{equation*}
\limsup_{y\rightarrow x}\dfrac{|f(y)-f(x)-df^{\alpha}(x) \cdotp (\pi_{\alpha}(y)-\pi_{\alpha}(x))|}{d(x, y)} = 0\; \text{ for}\; \mu-\text{a.e}\; x\in X_{\alpha}.
\end{equation*}
\end{enumerate}
\end{thm}
Cheeger also  provided a uniform statement that covers many of the known nonembedding results.
\begin{thm}[Cheeger]
If a doubling p-Poincar\'{e} space $X$ admits a bi-Lipschitz embedding into some finite dimensional Euclidean space, then at almost every point $x \in  X_{\alpha}$, the tangent cone of $X$ at $x$ is bi-Lipschitz equivalent to $\mathbb R^{N(\alpha)}$. 
\end{thm}
We can deduce from Cheeger's theorem the known nonembedding results both for the Carnot groups and for Laakso spaces.  Cheeger and Kleiner generalized  the almost everywhere differentiability  for Lipschitz maps on PI space to any Banach space $V$ with  Radon-Nikod\'{y}m property (\cite{CK2}, \cite{CK3}).

We now check nonembeddability of the Heisenberg group $\mathbb H$ by applying Cheeger's nonembedding theorem.
The Heisenberg group has a strong measurable differentiable structure with a single coordinate patch $(\mathbb{H}, \pi_1, \pi_2)$, where $\pi_1(x,y,t)=x$ and  $\pi_2(x,y,t)=y$.  If we assume that the Heisenberg group admits a bi-Lipschitz embedding into some  Euclidean space, then every tangent cone at almost every point in $\mathbb {H}$ must be  bi-Lipschitz equivalent to $\mathbb{R}^2$. Since the  Hausdorff dimension of $\mathbb H$ is not equal to  2, we conclude bi-Lipschitz nonembeddability.

In contrast to the Heisenberg group, Cheeger's nonembedding theorem does not answer whether or not the Grushin plane locally embeds into some Euclidean space. The Grushin plane $\mathbb{G}$ with Lebesgue measure is a doubling metric measure space supporting $p$-Poincar\'{e} inequality for any $p\geq 1$ (see Remark \ref{Grushinrem}).  Let $K$ be any compact subset of $\mathbb G$ and $\mathbb A$ be set of singular points, $y$-axis. It has a Cheeger's coordinate patch $(K\setminus \mathbb A, \pi_1, \pi_2)$, where $\pi_1(x,y)=x$ and  $\pi_2(x,y)=y$. Since every tangent cone to $K \setminus \mathbb A$ is bi-Lipschitz equivalent to $\mathbb{R}^2$, we cannot conclude non-embeddability of the Grushin plane, unlike the case of the Heisenberg group.  Indeed, we prove that the Grushin plane admits a bi-Lipschitz embedding into some Euclidean space.
\subsection{The Grushin plane}
The metric on $\mathbb{G} \setminus \mathbb A$ is the Riemannian metric $ds^2$ making $X_1$ and $X_2$   into an orthonormal basis for the tangent space,
\begin{equation}
ds^2= dx^2 +  \dfrac{dy^2}{x^{2}}.
\end{equation}
The metric can be extended across $\mathbb A$  as the Carnot-Carath\'{e}odory distance (or ${cc}$-distance ) by means of the length element $ds^2$, since the horizontal distribution satisfies the  H\"{o}rmander condition.

For any horizontal curve $\gamma : [0,1] \rightarrow \mathbb {G}$, we write $\gamma(t) = (\x(t),\,\y(t))$ for a parametrized horizontal curve. Then, we have
\begin{equation}
\label{metric}
\text{length}(\gamma)= \int_{0}^{1}   \sqrt{  {{\x'}(t)}^2  +   { \dfrac { { {\y'}(t)}^2 }   {  \x(t)^{2}         } }}  \,dt.
\end{equation}
The following proposition gives distance estimates for the Carnot-Carath\'{e}odory distance on  $\mathbb {G}$. 
\begin{prop}
The  Carnot-Carath\'{e}odory distance on $\mathbb A$ is comparable to $\sqrt{d_E}$. Now fix points  $p=(x_1, \,y_1)$ and $q=(x_2, \,y_2)$ in $\mathbb {G} \setminus \mathbb A$. We have the following distance estimates:
\begin{equation}
\label{G2}
\frac{1}{2}\left(|x_1-x_2|+\dfrac{|y_1-y_2|}{ \sqrt{{\min(|x_1|,|x_2|)}^2+4|y_1-y_2|}}\right)\leq d_{cc}(p,\,q) \leq  4(|x_1-x_2|+\sqrt{|y_1-y_2|}). 
\end{equation}
\end{prop}

\begin{proof}
The first estimation of the ${cc}$-distance on $\mathbb A$ is deduced from equation \eqref{dilation}.
The upper bound in equation \eqref{G2} comes from the triangle inequality. We will use equation \eqref{metric} to get the lower bound in equation \eqref{G2}. Let $\gamma(t) = (\x(t), \y(t))$ be a parametrized horizontal curve joining $p$ to $q$ where $t \in [0,1]$. Then,
\begin{equation*}
(*) \;\text{length}(\gamma) \geq |{x_1}-{x_2}|.
\end{equation*} 
 If there exists $K$ such that $|\x(t)| \leq K$ for all $t \in [0,1]$, then length($\gamma$) $\geq K^{-1}|{y_1}- {y_2}|$. Otherwise,  there exists $t_0 \in [0,1]$ such that $|\x(t_0)|\geq K$. Then,
 \begin{equation*}
  \text{length}(\gamma) \geq \text{ length}(\widetilde{\gamma}) \geq \max \{ |\x(t_0)-{x_1}|, |\x(t_0)-{x_2}|\} \geq K- \min\{|{x_1}|, |{x_2}|\} 
\end{equation*}  
  where $\widetilde{\gamma}$ is a subcurves of $\gamma$ joining $p$ to $(\x(t_0),\, \y(t_0))$ or $q$ to $(\x(t_0),\,\y(t_0))$. Then , we have the following:
\begin{equation*}
(**)\;\text{length}(\gamma) \geq \sup_{K> \max\{|{x_1}|, |{x_2}|\} } \min \{ K- \min\{|{x_1}|, |{x_2}|\}, \; K^{-1} |{y_1}-{y_2}|\}.
\end{equation*}
When we choose $K=\sqrt{{\min(|x_1|, |x_2|)}^2+4|y_1-y_2|}$ and average $(*)$ and $(**)$, then we get  distance estimates \eqref{G2}.
\end{proof}
We next consider the lattice of points in $\R^2 $ whose coordinates are integers. Then, this lattice determines a mesh $M_0 \times M_0$. For each $j \in \Z$, consider the submesh $M_j = 2^{-j}M_0\times 2^{-2j}M_0$ which is set of cubes in $\R^2$ of sidelengths $2^{-j}$  and  $2^{-2j}$  respectively. From the above distance estimates, $\mathbb {G} \setminus \mathbb A$ has a Whitney decomposition. We recall this in the following Proposition \ref{p}.
\begin{prop}
\label{p}
Let $\mathbb A$ be $y$-axis. Then its complement $\Omega =\mathbb{G}\setminus \mathbb A$ is the union of a sequence of cubes $Q$ , whose interiors are mutually disjoint and whose diameters are approximately proportional to their distances from $\mathbb A$. More precisely,
\begin{enumerate}
\item $\Omega = \cup_{Q \in W_{\Omega}}Q$.
\item Any two cubes are mutually disjoint.
\item $\dist_{cc}(Q, \mathbb A) \leq \diam_{cc}(Q) \leq 8 \dist_{cc}(Q, \mathbb A)$.
\end{enumerate}
\end{prop}
 The Grushin plane is  complete, doubling, and uniformly perfect metric space. Since ${cc}$-distance on $\mathbb A$ is comparable to $\sqrt{d_E}$,  we apply Assouad's theorem. Then we have a $L$-bi-Lipschitz embedding $f$ from $\mathbb A$ into $\mathbb R^3$  for some $L$.  If we verify the condition of uniformly Christ-local bi-Lipschitz embeddings, then we can conclude the following theorem.
\begin{thm}
\label{Gru}
The  Grushin plane equipped with Carnot-Carath\'{e}odory distance admits a bi-Lipschitz embedding into some Euclidean space.
\end{thm}
It is enough to verify  the existence of uniformly Christ-local bi-Lipschitz embeddings. In this case,  $Q^*$ is  the set of all Whitney cubes which touch $Q$ and $Q^{**}$  is the set of all Whitney cubes which touch $Q^*$ (see Definition \ref{Q^*}). 
\begin{lemma}
The complement of $\mathbb A$ admits uniformly Christ-local bi-Lipschitz embeddings.
\end{lemma}
\begin{proof}
 We observe that $Q^{**}$ is a closed 2-dimensional Riemannian manifold for each $Q$.   For any two elements $Q$ and $Q'$ in $\W_{\Omega}$, we have $Q' = \Phi(Q)$ where $\Phi$ is composition of translation map $\varsigma $ with respect to $\{0\}\times \mathbb R$ and expansion map $\psi (x,\,y)=(2^{(j'-j)}x,\, 2^{2(j'-j)}y)$. Then,  we have $\diam(Q')=2^{(j'-j)}\diam(Q)$  from Proposition \ref{p}. 
Therefore, we can cover all $Q^{**}$ by balls $B_1, \;B_2,\cdots,\; B_N$ of radius $\diam(Q) >0$ where $N$ is independent of $Q$.  For each $i$, there exist $L$-bi-Lipschitz diffeomorphisms for some $L$
\begin{equation*}
\varphi_i: 5B_i  \rightarrow \varphi_i(5B_i) \subset\mathbb R^{2}.
\end{equation*}
Without loss of generality, we may assume that $|\varphi_i(x)| \geq \diam(Q)$ for all $i$ and $x\in 5B_i$.
let $u_i \in C_0^{\infty}(2B_i) $ be such that $0 \leq u_i \leq 1$ and $u_i|_{B_i}=1$, and let $v_i \in C_0^{\infty}(5B_i) $ be such that $0 \leq v_i \leq 1$ and $v_i|_{4B_i}=1$. Then, we define $ \varphi:X \rightarrow \mathbb R^{2N}\times\mathbb R^{2N}$
\begin{equation*}
\varphi(x) :=(\varphi_{1}(x)u_1(x),\, \cdots ,\, \varphi_{N}(x)u_N(x) ,\, \varphi_1(x)v_1(x),\,\cdots,\, \varphi_N(x)v_N(x))
\end{equation*}
Obviously $\varphi$  is smooth, and hence it is Lipschitz with Lipschitz constant $2LN$. We will show that $\varphi$ is co-Lipschitz.
To this end, let us assume first that $d(x,y)>3\,\diam(Q)$. Then, there exists $i$ such that $u_i(x)=1$ and $v_i(x)=0$. Thus,
\begin{align*}
|\varphi(x)-\varphi(y)| &\geq |\varphi_i(x)u_i(x)-\varphi_i(y)u_i(y)|\\
                                   &=|\varphi_i(x)|\\
                                   &\geq \diam(Q) \geq\frac{1}{C(C_1,\, A,\, \delta)}\, {d(x,y)}
\end{align*}
The last inequality arises from  comparability of $\diam (Q)$ and $\diam(Q^{**}) $ ( see Proposition \ref{comparability} ).  On the other hand, if $d(x,y) \leq 3\diam(Q)$, then there exists $i$ such that $v_i(x)=1=v_i(y)$. Thus,
\begin{equation*}
|\varphi(x)-\varphi(y)| \geq |\varphi_i(x)-\varphi_i(y)| \geq \frac{1}{L}\,d(x,\,y).
\end{equation*}
Therefore, we have uniformly local  bi-Lipschitz embeddings on each $Q^{**}$ into $\R^{4N}$. The bi-Lipschitz constant and dimension of the target space are independent of $Q$.
\end{proof}
Theorem \ref{Gru} can be generalized to cover other singular sub-Riemannian manifolds similar to the Grushin plane.
We  denote points in $\R^n\times \R^l$ by $p=(x,\,y)$, where $x=(x_1,\, x_2,\,\cdots,\,x_n) \in \R^n$ and  $y=(y_1,\, y_2,\,\cdots,\,y_l) \in \R^l$. We let $\alpha=(\alpha_1, \, \alpha_2,\, \cdots,\, \alpha_n)$ be an $n$-tuple of non-negative integers with length  $|\alpha|=\sum_{i=1}^n \alpha_i$. If $x=(x_1,\,x_2,\,\cdots,\,x_n) \in \R^n$, we put $x^{\alpha}:=x_1^{\alpha_1}x_2^{\alpha_2}\cdots x_n^{\alpha_n}$.
\begin{defn}
\label{GS}
The space of Grushin type is $\mathbb R^n \times \mathbb R^l$ for  $n,\, l \in \mathbb N$ with horizontal distribution spanned by
$X_i$ and $Y_j$ for  $i=1,\;2,\;,\cdots,\; n$ and  $j=1,\;2,\;,\cdots,\; l$
\begin{equation*}
X_i=\dfrac{\partial}{\partial x_i}\;\;\;\; \text{and} \;\;\;\;Y_j=x^{{\alpha}^j  } \dfrac{\partial }{\partial y_j}
\end{equation*}
where  for each $j$, ${\alpha ^j}$ is an $n$-tuple of non-negative integers ${\alpha}_i^j$  and  $| {\alpha}^j| = k$. 
\end{defn}
\begin{thm}
The space of Grushin type equipped with Carnot-Carath\'{e}odory distance admits a bi-Lipschitz embedding into some Euclidean space.
\end{thm}
\begin{proof}
We can follow similar steps as the proof of Theorem \ref{Gru}. We omit the proof.
\end{proof}
\section{Questions and Remarks}
So far we have given a characterization of  Euclidean bi-Lipschitz embeddability of uniformly perfect metric spaces supporting a doubling measure. The hypothesis in Theorem \ref{main} is based on a Christ-Whitney decomposition deduced from uniform perfectness and existence of a doubling measure. We emphasize that  uniform perfectness is only used for existence of a Christ-Whitney decomposition.

\begin{que}
Can the condition of  uniform perfectness be weakened?
\end{que}

From Theorem \ref{main}, the dimension $M_1+M\,M_2+1 $ of the Euclidean space depends on the bi-Lipschitz constant $L_1$ and the doubling constant of $\mu$.  However, the number of colors $M$ is not optimal. Thus, the following question naturally arises.
\begin{que}
What is the minimal dimension of Euclidean space into which a metric space satisfying the conditions in Theorem \ref{main} bi-Lipschitzly embeds?
\end{que}

As an application of Theorem \ref{main}, we have considered  the Grushin space. We now can consider 
the space of Grushin type with \emph{extended} horizontal distribution on $\R^n \times \R^l$.
\begin{defn}
The \emph{extended space of Grushin type} is $\mathbb R^n \times \mathbb R^l$ for  $n,\, l \in \mathbb N$ with horizontal distribution spanned by
$X_i$ and $Y_j$ for  $i=1,\;2,\;,\cdots,\; n$ and  $j=1,\;2,\;,\cdots,\; l$
\begin{equation*}
X_i=\dfrac{\partial}{\partial x_i}\;\;\;\; \text{and} \;\;\;\;Y_j=x^{{\alpha}^j  } \dfrac{\partial }{\partial y_j}.
\end{equation*}
\end{defn}
\begin{rem}
We emphasize that the lengths $|{\alpha}_i^j|$ can be distinct in Definition \ref{GS}. 
\end{rem}
 It seems that we can follow similar steps to prove embeddability. However, some of the technical details must be checked.
\begin{conje}
The \emph{extended space of Grushin type} equipped with Carnot-Carath\'{e}odory distance admits a bi-Lipschitz embedding into Euclidean space.
\end{conje}
In the case of spaces of Grushin type,  horizontal distributions are good enough to have uniformly Christ-local embeddings. Therefore, the following problem naturally comes up.
\begin{prob}
Find sufficient conditions on a higher dimensional horizontal distribution in a given sub-Riemannian manifold so as to guarantee the existence of uniformly Christ-local bi-Lipschitz embeddability.
\end{prob}
Even more generally,  we meet the following problem:
\begin{prob}
\label{que515}
Characterize  Christ-local bi-Lipschitz embeddability.
\end{prob}
If Problem \ref{que515} were solved, then we could  characterize bi-Lipschitz embeddable metric spaces with geometric and analytic criteria.  Therefore, we could determine which metric spaces admit a bi-Lipschitz embedding and we can classify up to bi-Lipschitz equivalence those metric spaces which are subsets of Euclidean space.
\begin{prob}
Find other examples of sub-Riemannian manifolds that satisfy conditions in Theorem \ref{main}  and hence, embed bi-Lipschitzly into Euclidean space. 
\end{prob}

 \bibliographystyle{plain}
\bibliography{PAPER.bib}
\end{document}